\documentclass[reqno, 11 pt]{amsart}
\usepackage[T1]{fontenc}

\usepackage{marginnote}
\usepackage[center]{caption}

\usepackage{mathrsfs}
\usepackage{bbm}

\usepackage{amsthm} %%% Defines \theoremstyle
\usepackage{amsmath} %%% Defines \numberwithin, \operatorname
\usepackage{amssymb} %%% Defines \mathbb
\usepackage{mathtools} %%% Defines \eqqcoon and lots of symbols and arrows

\usepackage{graphicx}
\usepackage{xspace}
\usepackage[all]{xy}
\usepackage{pinlabel}
\usepackage{enumerate}
\usepackage[margin=.9 in, marginparwidth=.7 in, marginparsep=.1 in]{geometry}

\usepackage{tikz}

\usetikzlibrary{cd, calc,through,backgrounds, arrows}
\usetikzlibrary{decorations.pathmorphing, decorations.pathreplacing,positioning,fit}

\usepackage{hyperref}

\makeatletter

\newcommand{\gap}{\,\,\,\,\,\,}

\newcommand{\calA}{\mathcal{A}}
\newcommand{\calB}{\mathcal{B}}
\newcommand{\calC}{\mathcal{C}}

\newcommand{\calG}{\mathcal{G}}
\newcommand{\calH}{\mathcal{H}}

\newcommand{\calK}{\mathcal{K}}

\newcommand{\calP}{\mathcal{P}}

\newcommand{\HH}{\mathbb{H}}

\newcommand{\NN}{\mathbb{N}}

\newcommand{\RR}{\mathbb{R}}

\newcommand{\ZZ}{\mathbb{Z}}

\newcommand{\bfb}{\mathbf{B}}

\newtheorem{theorem}{Theorem}[section]
\newtheorem{proposition}[theorem]{Proposition}
\newtheorem{corollary}[theorem]{Corollary}
\newtheorem{lemma}[theorem]{Lemma}

\newtheorem{introthm}{Theorem}

\theoremstyle{definition}
\newtheorem{definition}[theorem]{Definition}

\newtheorem*{claim*}{Claim}

\newtheorem*{question*}{Question}
\newtheorem*{answer*}{Answer}
\newtheorem*{application*}{Application}

\theoremstyle{remark}
\newtheorem{remark}[theorem]{Remark}
\newtheorem*{remark*}{Remark}

\newcommand{\cat}{\operatorname{CAT}}

\newcommand{\diam}{\operatorname{diam}}

\newcommand{\Ham}{Ham\-en\-st\"adt\ }
\newcommand{\barchi}{\bar{\chi}}
\newcommand{\barphi}{\bar{\phi}}

\newcommand{\param}{{\mathchoice{\mkern1mu\mbox{\raise2.2pt\hbox{$
\centerdot$}}
\mkern1mu}{\mkern1mu\mbox{\raise2.2pt\hbox{$\centerdot$}}\mkern1mu}{
\mkern1.5mu\centerdot\mkern1.5mu}{\mkern1.5mu\centerdot\mkern1.5mu}}}
\DeclarePairedDelimiter\norm{\lVert}{\rVert}

\renewcommand{\setminus}{{\smallsetminus}}

\newcommand{\from}{\colon\thinspace}

\newcommand{\id}{\operatorname{id}}
\newcommand{\gamhat}{\hat{\gamma}}
\newcommand{\lbdry}{\partial_\infty L}
\newcommand{\ham}{\text{Ham}}
\newcommand{\barx}{{\bar{x}}}
\newcommand{\bary}{{\bar{y}}}
\newcommand{\barh}{\bar{h}}
\newcommand{\bdry}{\partial_\infty \mtilde}
\newcommand{\fix}{\operatorname{Fix}}

\newcommand{\clowval}{\frac{n-1}{n-2}}
\newcommand{\dhatuu}{\hat{d}^{\text{uu}}}
\newcommand{\diamhat}{\widehat{\text{diam}}}
\newcommand{\suu}{S^{\text{uu}}}  
\newcommand{\wuu}{W^{\text{uu}}}
\newcommand{\buu}{\text{B}^{\text{uu}}}
\newcommand{\duu}{d^{\text{uu}}}
\newcommand{\muu}{\mu^{\text{uu}}}
\newcommand{\mubuu}{\mu_B^{\text{uu}}}

\newcommand{\sss}{S^{\text{ss}}}

\newcommand{\kpahat}{\hat{\kappa}}
\newcommand{\bfzeta}{\boldsymbol{\zeta}}
\newcommand{\bfeta}{\boldsymbol{\eta}}
\newcommand{\conj}{\operatorname{Conj}}
\newcommand{\clow}{{\mathfrak{b}}}

\newcommand{\test}[3]{\phi_{#1}^{#2} (#3)}
\newcommand{\bartest}[3]{\barphi_{#1}^{#2} (#3)}

\newcommand{\pdf}{p_D^F}

\newcommand{\cte}{\frac{\sigma(F_1) \sigma(F_2)}{\delta}}

\newcommand{\mtilde}{\widetilde{M}}
\newcommand{\lip}[1]{$#1$--Lipschitz}
\newcommand{\hold}[1]{$#1$--H\"{o}lder}

\newcommand{\mBM}{\widetilde{m}_{\text{BM}}}
\newcommand{\mbarBM}{m_{\text{BM}}}

\newcommand{\wss}[1]{W^{\text{ss}} (#1)  }
\newcommand{\ws}{W^{\text{s}} }

\newcommand{\muss}{\mu^{\text{ss}} }
\newcommand{\mus}{\mu^{\text{s}} }

\newcommand{\mub}[1]{\mu^{\text{ss}}_\text{B} \left( #1 \right) }
\newcommand{\bss}[1]{\text{B}^{\text{ss}}  (#1)  }

\newcommand{\thickbs}[2]{\mathbf{B}^{\text{s}}_{#1} \left( #2 \right) }

\newcommand{\pss}{p^\text{ss}}

\newcommand{\dss}{d^{\text{ss}} }

\newcommand{\tauss}{\tau^\text{ss}_F}
\newcommand{\rss}{r^\text{ss}_F}

\newcommand{\gmacycle}{ \langle \gamma \rangle }

\newcommand{\gamhatcycl}{ \langle \gamhat \rangle }

\newcommand{\bfO}{\mathbf{O}}

\newcommand{\holder}[1]{H\"{o}lder}
\newcommand{\gmcycl}{\langle \gamma \rangle }
\newcommand{\stab}[1]{\operatorname{Stab}(#1) }

\newcommand{\sprt}{\text{Supp}}
\newcounter{counterc}
\newcounter{counteralpha}
\newcommand{\nextalfa}[1]{\refstepcounter{counteralpha} \label{#1}}
\newcommand{\alfa}[1]{\ensuremath{  {\alpha_{#1 }} } }

\newcounter{counterkpa}
\newcommand{\nextkpa}[1]{\refstepcounter{counterkpa} \label{#1} }
\newcommand{\kpa}[1]{\ensuremath{  {\kappa_{#1 }} } }
\newcounter{counterconst}
\newcommand{\nextconst}[1]{\refstepcounter{counterconst} \label{#1} }
\newcommand{\const}[1]{{\ensuremath{c_{#1 } }} }

\newcommand{\rtaupm}[1]{
    \ensuremath{ 
        r #1 \const{\ref{thick-const}} \epsilon^\alfa{\ref{thick-alpha}}, 
        \tau #1 \const{\ref{thick-const}} \epsilon^\alfa{\ref{thick-alpha}}
    }
}

\newcommand{\chilowerdelt}{
    \const{\ref{chiplus const}} \frac{\epsilon^\alfa{\ref{chiplus-alpha}}}{\tau}
}

\title{The growth of a fixed conjugacy class in negative curvature}
\author{Pouya Honaryar}

\begin{document}

\maketitle

\begin{abstract}
    Let $M$ be a compact closed manifold of variable negative curvature. Fix an element $\id \neq \gamma$ in the fundamental group $\Gamma$ of $M$, and denote the set of elements in $\Gamma$ that are conjugate to $\gamma$ by $\conj_\gamma$. For two points $x, y$ in the universal cover of $M$, we obtain asymptotics for the number of $\conj_\gamma$--orbits of $y$ that lie in a ball of radius $T$ centered at $x$, as $T$ tends to infinity. If $M$ is two-dimensional, or of dimension $n \geq 3$ and curvature bounded above by $-1$ and below by $-(\clowval)^2$, we find an exponentially small error term for this count.
\end{abstract}

% \include{Part_1.tex}

%Until the beginning of counting:

\section{Introduction} \label{intro}

% \subsection{Statement of results}
{\bf Statement of the main result.}
Let $M$ be a closed compact manifold of (variable) negative curvature, and let $\mtilde$ be its universal cover. The fundamental group of $M$, denoted by $\Gamma$, acts on $\mtilde$ by deck transformations. We denote the image of a point $x \in \mtilde$ under the action of $g \in \Gamma$ by $g.x$, or simply by $gx$, and denote the $\Gamma$ orbit of $x$ by $\Gamma \cdot x$. We denote the ball of radius $T$ centered at a point $x$ by $B_T(x)$. In his Ph.D. thesis (see \cite{margulis-thesis} for an english translation), Margulis proved that for $x, y \in \mtilde$,
\begin{align} \label{margulis-count}
    \# (B_T(x) \cap \Gamma \cdot y) \sim c_{x, y} e^{\delta T}, \qquad \text{ as } T \to \infty,
\end{align}
where $c_{x, y}$ is a constant only depending on $x$ and $y$, and $\delta$ denotes the topological entropy of geodesic flow on the unit tangent bundle of $M$.

In \cite{pollicott_error}, using the dynamical zeta function approach, Sharp and Pollicott obtained exponentially small error term for the above count, when $M$ is $2$--dimensional and $x = y$. More precisely, for a closed surface $M$ of variable negative curvature, they proved there exists a constant $\kappa > 0$, only depending on the geometry of $M$, such that for $x \in \mtilde$ we have 
\begin{align*}
    \# (B_T(x) \cap \Gamma \cdot x) = c_{x, x} e^{\delta T} + \bfO(e^{(\delta - \kappa) T}).
\end{align*}

Fixing an element $\id \neq \gamma \in \Gamma$, denote the conjugacy class of $\gamma$ by $\conj_\gamma$, that is, 
$$\conj_\gamma \coloneqq \{g^{-1}\gamma g: g\in \Gamma \}.$$
Our main theorem, proved in Section \ref{count-control}, is as follows.

\begin{introthm} \label{main-theorem}
    Assume $M$ is two-dimensional, or of dimension $n \geq 3$ and curvature bounded above by $-1$ and below by $-(\clowval)^2$. Then there exists a constant $\kappa$, only depending on the geometry of $M$, such that the following holds.
    For every $\id \neq \gamma \in \Gamma$ and $x, y \in \mtilde$, there exists a constant $\sigma = \sigma(x, y, \gamma)$ such that for $T > 0$,
    \begin{align*}
        \#(B_T(x) \cap \conj_\gamma \cdot y) = \sigma e^{\frac{\delta}{2}T} + \bfO_{\gamma, x, y} (e^{(\frac{\delta}{2} - \kappa)T}).
    \end{align*}
    % where $F_1^\gamma$ and $F_2^\gamma$ are as in ??.
\end{introthm}

\medskip{}
{\bf Previous works.}
The above theorem was first proved in \cite{huber_56} when $M$ is a compact surface of constant curvature $-1$, using Selberg's trace formula. (See also \cite{huber_98}) In \cite{poulin_conj}, the same is proved when $M$ is a compact manifold of constant curvature $-1$ and arbitrary dimension. 
% They also proved it for certain non-compact manifolds of constant curvature $-1$ (see Theorem 8 of that paper). 
In Corollary 10 of the same paper, for $M$ a compact manifold of variable negative curvature, it is proved that
$$\#(B_T(x) \cap \conj_\gamma \cdot y)  \asymp_{x, y, \gamma} e^{\frac{\delta}{2} T},$$
that is, there exists a constant $c= c(x, y, \gamma)$, such that for $T> 0$,
$$\frac{1}{c} e^{\frac{\delta}{2} T} \leq 
\#(B_T(x) \cap \conj_\gamma \cdot y) \leq c e^{\frac{\delta}{2} T}.$$

In the preprint \cite{pollicott_conj_prep}, building on \cite{pollicott_comparison}, Pollicott obtained asymptotics (without an error term) for $\conj_\gamma$ orbit counts, when $M$ is a surface of variable negative curvature \emph{with boundary}. At the end of that paper it is mentioned that the methods of the paper are not likely to work for a closed surface (i.e., a surface without boundary).

\medskip{}

{\bf Adjusted counts.}
The proof of Theorem \ref{main-theorem} has two main ingredients, the first being Proposition \ref{reduction-prop}, and the second being Theorem \ref{control counting}. Assuming these two, the proof of Theorem \ref{main-theorem} is given in Section \ref{count-control}. Theorem \ref{control counting} is stated for a point $x \in \mtilde$ and the axis $L_\gamma$ of an element $\id \neq \gamma \in \Gamma$, however, a similar statement holds (with the exact same proof) in the case where $L_\gamma$ is replaced by a point $y \in \mtilde$. (See Remark (\ref{convex_gen})) In this introduction, to better illustrate the idea beneath this thoerem, we state it for this latter case as Thoerem \ref{adjust-pp}. 
% consists of two steps.
% In the first step we express the distance between two points, both close to the boundary of $\mtilde$, in terms of a time parameter $t$ and several bounded Busemann functions. In the second step, we prove a generalization of Theorem 3 of \cite{PPCommonPerp}, which helps us conclude the proof. 
% % After introducing some notation, more details are given at the end of Section \ref{background}, in which we give a proof of Theorem \ref{main-theorem}, assuming ?? and ??.
% To explain this generalization, let us introduce some notation. Let $x$ be a point in $M$, and $L$ a closed geodesic in $M$. A \emph{perpendicular} $\alpha$ from $x$ to $L$ is a locally geodesic path that starts from $x$ and arrives perpendicularly at $L$. Denote the set of such perpendiculars by $\prp_{x, L}$. 

Fix $x, y \in \mtilde$ and let $h \from \Gamma \to \RR$ be a function such that for $g \in \Gamma$,
$$h(g) = d(x, gy) + \bfO(1).$$
Setting
$$\calB^h (T) \coloneqq \{g \in \Gamma : h(g) \leq T \},$$
it is clear that $\# \calB^h (T) \asymp e^{\delta T}$; however, to find more precise estimates on $\# \calB^h (T)$ we need more control over the difference $h(g) - d(x, gy)$. Theorem \ref{adjust-pp} states that if this difference only depends (upto a small error term) on the manner the geodesic segment $[x, gy]$ departs from $x$ and the manner it arrives at $gy$, then asymptotics for $\# \calB^h (T)$ can be obtained. Moreover, under some additional assumptions, an exponentially small error term may be obtained for this count. To state Theorem \ref{adjust-pp}, let $S(x)$ denote the unit tangent sphere based at $x \in \mtilde$, and for $g \in \Gamma$, let $v_x(g) \in S(x)$ (resp. $v_y(g) \in S(gy)$) be the tangent to $[x, gy]$ at $x$ (resp. $gy$) pointing towards $gy$ (resp. $x$).

% and let $\partial^1 L$ denote the unit normal bundle of $L$. For $\alpha \in \prp_{x, L}$, set $v_x(\alpha) \in S(x)$ and $v_L(\alpha) \in \partial^1 L$ to be the unit vectors tangent to $\alpha$ at its endpoints. Let $F_x$ (resp. $F_L$) be a real valued function on $S(x)$ (resp. $\partial^1 L$). For a perpendicular $\alpha$, denote the length of $\alpha$ by $\ell(\alpha)$, and define the \emph{adjusted length} of $\alpha$, with respect to the \emph{adjustment functions} $F_x$ and $F_L$, by
% $$\ell'(\alpha) \coloneqq \ell(\alpha) - F_x(v_x(\alpha)) - F_L(v_L(\alpha))$$
% Let
% \begin{align*}
%     \prp'_{x, L} (T) \coloneqq \{\alpha \in \prp_{x, L} : \ell'(\alpha) \leq T \}
% \end{align*}
% be the set of perpendiculars from $x$ to $L$ of adjusted length at most $T$. A slightly more general version of the following is proved as Theorem ??.

\begin{introthm} \label{adjust-pp}
    Assume $M$ is two-dimensional, or of dimension $n \geq 3$ and curvature bounded above by $-1$ and below by $-(\clowval)^2$. Then for every $\alpha, \kappa > 0$, there exists a constant $\kappa'$, only depending on $\alpha, \kappa$, and the geometry of $M$, such that the following holds.
    Let $x, y \in \mtilde$, and 
    \begin{align*}
        F_x \from S(x) \to \RR, \quad \text{ and }\quad F_y \from S(y) \to \RR
    \end{align*}
    be \hold{\alpha} functions. Let the function $h \from \Gamma \to \RR$ satisfy
    \begin{align} \label{h_eq}
        h(g) = d(x, gy) - F_x (v_x(g)) - F_y(g^{-1}. v_y (g)) + \bfO ({e^{-\kappa d(x, gy) } }) 
    \end{align}
    for $g \in \Gamma$. Then there exists a constant $\sigma = \sigma (F_x, F_y)$ such that 
    \begin{align*}
        \# \calB^h (T) = \sigma e^{\delta T} + \bfO_{F_1, F_2} (e^{(\delta  - \kappa') T}).
    \end{align*}
\end{introthm}

For $g \in \Gamma$, $h(g)$ can be thought of as the distance between $x$ and $gy$, adjusted by the functions $F_x$ and $F_y$ on both endpoints of $[x, gy]$. $F_x$ and $F_y$ are called the \emph{adjustment functions}, and $h(g)$ is called the \emph{adjusted distance} between $x$ and $gy$.

\medskip{}

{\bf Remarks.} We provide the following remarks for the experts.
\begin{enumerate} [(i)]
    \item \label{convex_gen} Our proof of Theorem \ref{adjust-pp} closely follows the proof of Theorem 3 of \cite{PPCommonPerp} (see the beginning of Section \ref{test-intro} for a discussion), hence we expect that Theorem \ref{adjust-pp} holds in the more general setting of \cite{PPCommonPerp}, that is, when $x$ and $y$ are replaced by nonempty proper closed convex subsets $C_1$ and $C_2$ of $\mtilde$ such that $\stab{C_i} \backslash C_i$ has finite skinning measure for $i = 1,2$; and the role of $[x, gy]$ is played by the common perpendicular between $C_1$ and $g . C_2$. Indeed, we refrained from writing the proof in this generality to first, keep the notation simple, and second, since we did not have a geometric incentive to do so.
    %%%%%%%%%%%%%%%%%%%%%%%%%%%%%%%%%%%%%%%%%%%%%%%%%%
    \item \label{poli_err} As mentioned at the beginning of this section, the error term for (\ref{margulis-count}) is obtained, when $x = y$ and $M$ is a compact surface, in \cite{pollicott_error}. The proof given in that paper relies on a certain coding of geodesic flow, wich is not established (to the best of the author's knowledge) for manifolds of dimension greater than $2$. We now describe an alternative approach to obtain error term. By Theorem \ref{radius-hold-cont}, the so-called RHC condition (see Definition \ref{rhc_def}) holds for manifolds satisfying the assumptions of Theorem \ref{main-theorem}, hence, using Theorem 3 of \cite{PPCommonPerp} for two points $x, y \in \mtilde$, we obtain exponentially small error term in (\ref{margulis-count}) for such manifolds. This is, to the best of the author's knowledge, the best result in this direction.
\end{enumerate}

\medskip{}
{\bf Notation and conventions.} 
Let $A, B \in \RR$ and $\bullet$ be a set of parameters. We write $A = \bfO_\bullet (B)$ if $|A| \leq c B$ for a constant $c = c(\bullet)$ that only depends on $\bullet$. 
We write $A \prec_\bullet B$ (resp. $A \succ_\bullet B$, $A \asymp_\bullet B$) if there exists a constant $c = c(\bullet)$ such that $A \leq c B$ (resp. $A \geq \frac{B}{c}$, $\frac{A}{c} \leq B \leq c A$).
For any of these symbols, if a subset of the parameters $\bullet$ are fixed (at the beginning of a section or throughout a proof), we may remove them from the subscript of that symbol.

\medskip{}
{\bf Acknowledgements.} I want to thank my advisor, Kasra Rafi, for his constant support during the writing of this paper. I also want to thank Fr\'{e}d\'{e}ric Paulin for helpful discussions on the RHC condition.

\section{Background} \label{background}

% The notation and statements in this section are standard. We refer the reader to Chapter 2 of \cite{pps-book} for more details.
We fix a closed compact manifold of (variable) negative curvature $M$ throughout the text, and denote its universal cover by $\mtilde$. We further assume that the curvature of $M$ is bounded from above (resp. below) by $-1$ (resp. $-\clow^2$).
We choose, once and for all, a point $o \in \mtilde$, called the \emph{origin}, which remains the same for the rest of this paper.
We denote the unit tangent sphere at $x \in \mtilde$ by $S(x)$, and we use the same notation when $x$ is an element of $M$. The unit tangent bundle of $\mtilde$ (resp. $M$) is denoted by $T^1 \mtilde$ (resp. $T^1 M$), and the \emph{basepoint projection} map, sending a vector in $T^1 \mtilde$  (resp. $T^1M$) to its basepoint in $\mtilde$  (resp. $M$), is denoted by $\pi$ in both cases.
The natural map from $\mtilde$ to $M$, denoted by $\Pi$, induces a map from $T^1 \mtilde$ to $T^1 M$, which we also denote by $\Pi$. For a point $x$ in $\mtilde$ (resp. $M$) and $u \in S(x)$, we denote $-u$ by $\iota(u)$. Flowing $u \in T^1 \mtilde$ (resp. $u \in T^1 M$) by time $t$, we obtain an element of $T^1 \mtilde$ (resp. $T^1 M$), which we denote by $\calG_t(u)$ or $u_t$, depending on the occasion. The fundamental group $\Gamma \coloneqq \pi_1(M)$ acts by deck transformations on $T^1 \mtilde$, and we denote the image of $u \in T^1 \mtilde$ under the action of $g \in \Gamma$ by $g.u$, or simply by $gu$ when no confusion can arise. It is well-known that if $\id \neq \gamma \in \Gamma$, then there exists a bi-infinite geodesic $L_\gamma$, called the \emph{axis} of $\gamma$, such that $\gamma$ acts on $L_\gamma$ by translation.

The boundary at infinity of $\mtilde$ is denoted by $\bdry$. 
For $x, y \in \mtilde$, the geodesic segment connecting $x$ to $y$ is denoted by $[x, y]$, and for $\zeta \in \bdry$, the geodesic ray from $x$ to $\zeta$ is denoted by $[x, \zeta)$. 
For $u \in T^1 \mtilde$, the geodesic ray $\{\pi(u_t): t \geq 0\}$ (resp. $\{\pi(u_t): t \leq 0\}$) hits $\bdry$ at a point which we denote by $u^+$ (resp. $u^-$).
For two distinct points $x, y \in \mtilde$, the unit vector $v \in S(x)$ (resp. $v \in S(y)$) such that $[x, y] = \{\pi(v_t): 0 \leq t \leq d(x, y)\}$ is called the (unit) tangent vector to $[x, y]$ at $x$ (resp. $y$), and is denoted by $P^1_x(y)$ (resp. $P^1_y(x)$). Fixing $x$, the map $P^1_x \from \mtilde \to S(x)$ extends continuously to a map from $\mtilde \cup \bdry$ to $S(x)$, which we also denote by $P^1_x$.

A bi-infinite geodesic line $L$, simply called a geodesic from now on, hits the boundary of $\mtilde$ at two points, the set of which we denote by $\partial_\infty L$. Fixing a geodesic $L$, we denote the foot of perpendicular from a point $x \in \mtilde$ to $L$ by $P_L(x)$, and denote the vector tangent to $[x, P_L(x)]$ at $P_L(x)$ by $P^1_L(x) \in \partial^1 L$, where $\partial^1 L$ denotes the unit normal bundle of $L$. The map $P_L \from \mtilde \to L$ (resp. $P^1_L: \mtilde \to \partial^1 L$) extends continuously to a map from $\mtilde \cup \bdry \setminus L_\infty$ to $L$ (resp. $\partial^1 L$), which we also denote by $P_L$ (resp. $P^1_L$). 
% Denoting the tangent 

% For $y \in \mtilde$, we denote  (resp. $P^1_x(y)$). This extends continuously to a map  where $\partial^1 L$ denotes the unit normal bundle of $L$. (resp. $P^1_x \from \mtilde \cup \bdry \to S(x)$). 
% For $x, 

For $x \in \mtilde$ and $\zeta, \eta \in \bdry$, define the \emph{visual distance} between $\zeta, \eta$ as seen by $x$, by
\begin{align*}
    d_x (\zeta, \eta) \coloneqq e^{-T} \quad \text{ for }\quad T \coloneqq \frac{1}{2} \lim_{t \to \infty} (d(\zeta_t, \eta_t) - d(x, \zeta_t)  - d(x, \eta_t)),
\end{align*}
where $t \mapsto \zeta_t$ (resp. $t \mapsto \eta_t$) is any geodesic ray ending at $\zeta$ (resp. $\eta$). One can easily see that for $x,\, \zeta,\, \eta,\, T$ as above, and $u, v \in S(x)$ such that $u^+ = \zeta$ and $v^+ = \eta$, we have
$$d(\pi(u_T), \pi(v_T)) \asymp_{\clow} 1.$$
The visual distance remains the same, upto a multiplicative constant, if we change the basepoint. More precisely, for $x, y \in \mtilde$ and $\zeta, \eta \in \bdry$ we have
\begin{align*}
    e^{-d(x, y)} \leq \frac{d_x(\zeta, \eta)}{d_y(\zeta, \eta)} \leq e^{d(x, y)}.
\end{align*}

The \emph{Busemann cocycle} $\beta \from \bdry \times \mtilde \times \mtilde \to \RR$ is defined by 
\begin{align*}
    \beta(\zeta, x, y) \coloneqq \lim_{t \to \infty} (d(\zeta_t, x) - d(\zeta_t, y)),
\end{align*}
where $t \mapsto \zeta_t$ is any geodesic ray ending at $\zeta$. We summarize the elementary properties of Busemann cocycle as follows. For $x, y, z \in \mtilde$ and $\zeta \in \bdry$, we have
\begin{align*}
    \beta(\zeta, x, y) &= -\beta(\zeta, y, x);\\
    \beta(\zeta, x, y) &= \beta(\zeta, x, z) + \beta(\zeta, z, y);\\
    |\beta(\zeta, x, y)| &\leq d(x, y).
\end{align*}
% We define $\beta_0(\zeta, x) \coloneqq \beta(\zeta, x, o)$, and think of it as the `length' of the geodesic ray $[x, \zeta)$.

We consider the following `default' metrics.
\begin{itemize}
    \item The Riemannian metric $d_{\mtilde}$ (resp. $d_M$) on $\mtilde$ (resp. $M$).
    \item The Sassaki metric $d_S$ on $T^1 \mtilde$ (resp. $T^1 M$), which is the Riemannian metric induced by Sassaki's inner product on $TT\mtilde$. (resp. $TT M$.) See Section 2.3 of \cite{pps-book} for more details.
    \item The visual distance $d_o$ on $\bdry$, seen from the origin $o$. 
\end{itemize}
We may remove the subscript from $d_\bullet$ when it can be infered from the context.

For metric spaces $(X, d_X)$ and $(Y, d_Y)$, $X \times Y$ is always equipped with the sup metric 
$$d_{X \times Y} ((x_1, y_1), (x_2, y_2)) \coloneqq \max \{d_X(x_1, x_2), d_Y(y_1, y_2) \}.$$ 
A function $F \from X \to Y$ is said to be \hold{(\alpha, c)} for some constants $\alpha, c > 0$, if for all $x_1, x_2 \in X$ we have
\begin{align*}
    d_Y(F(x_1), F(x_2)) \leq c d_X(x_1, x_2)^\alpha.
\end{align*}
We say that the function $F$ is \hold{\alpha} for some $\alpha > 0$, if for every compact set $K \subset X$ there is a constant $c_K$ such that the restriction of $F$ to $K$ is \hold{(\alpha, c_K)}. Finally, we say that $F$ is \holder{}-continuous if it is \hold{\alpha} for some $\alpha > 0$. Note that most combinations of \holder{}-continuous functions are \holder{}-continuous. For example, if $X, Y_1, Y_2, Z$ are metric spaces and $F_1 \from X \to Y_1,\, F_2 \from X \to Y_2,\, H \from Y_1 \times Y_2 \to Z$ 
are \holder{}-continuous, then $x \mapsto H(F_1(x), F_2(x))$ is also \holder{}-continuous.

\section{\holder{}-continuous functions} \label{hold_section}
The goal of this section is to prove various functions, defined in terms of geometry of $\mtilde$, are \holder{}-continuous. These facts are used throughout the text, but especially in Section \ref{reduction-control}, to establish the \holder{}-continuity of functions introduced in Proposition \ref{reduction-prop}. The first few lemmas are well-known; we provided proofs however, since we couldn't find any in the literature. 
% We fix the origin $o \in \mtilde$ throughout this section and the next, and we do not show the dependence of implied constants on $o$. (See the notation and conventions introduced at the end of Section \ref{intro}.)

\begin{lemma} \label{gt-holder}
    The map from $\RR \times T^1 \mtilde$ to $T^1 \mtilde$ sending $(t, u)$ to $\calG_t(u)$ is Lipschitz.
    % $\calG(\param, \param)$ is \holder{}.
\end{lemma}

\begin{proof}
    Fixing $u \in T^1 \mtilde$, one can directly check that $\calG_t(u)$ is \lip{1} in $t$ (in fact, it is distance preserving). Fixing $t$, $\calG_t(\param)$ (the time $t$ flow on $T^1 M$) is a diffeomorphism on the compact manifold $T^1 M$, thus it is Lipschitz with respect to the Riemannian metric $d_S$ on this manifold. It follows that $\calG_t(\param)$ (the time $t$ flow on $T^1 \mtilde$) is also Lipschitz.
\end{proof}

\begin{lemma} \label{same endpoint converge}
    \begin{enumerate}[(i)]
        \item \label{all in mtilde} Fix $A > 0$; let $x, y, z \in \mtilde$ be such that $d(z, x) = d(z, y)$, and let $x' \in [z, x]$ and $y' \in [z,y]$ be such that $d(x', x) = d(y', y) = T$. Then 
        $$d(x', y') \prec_A e^{-T} d(x, y).$$
        
        \item \label{one in boundary} The conclusion of part (\ref{all in mtilde}) still holds if $z \in \bdry$ and $\beta(z, x, y) = 0$.

        \item \label{converge coarse} Fix $A > 0$, and let $x, y, z \in \mtilde$ be such that $d(x, y) \leq A$, and let $x' \in [z, x]$. Then, there exists $y' \in [z, y]$ such that
        $$d(x', y') \prec_A e^{-d(x', x)}\, d(x, y).$$ 

        \item \label{one in boundary coarse} Part (\ref{converge coarse}) remains true if $z \in \bdry$.
    \end{enumerate}
\end{lemma}

\begin{proof}
    Part (\ref{all in mtilde}) of this lemma directly follows from part (i) of Lemma 2.5. of \cite{pps-book}. Part (\ref{one in boundary}) can be proved in a similar way.
    To prove (\ref{converge coarse}), slide $y$ along the geodesic ray starting from $z$ and passing through $y$ to obtain $y''$ with $d(z, x) = d(z, y'')$.
    By trianlge inequality we have 
    $$d(y, y'') = |d(z, y) - d(z, x) | \leq d(x, y),$$
    hence
    $$d(x, y'') \leq d(x, y) + d(y, y'') \leq 2d(x, y) \leq 2A.$$
    Let $T \coloneqq d(x', x)$, and let $y' \in [y'', z]$ be such that $d(y'', y') = T$. By part (\ref{all in mtilde}) we have $d(x', y') \prec_A e^{-T} d(x, y)$, which is what we wanted. (\ref{one in boundary coarse}) follows from (\ref{one in boundary}) by a similar argument.
\end{proof}

In part (4) of Proposition 2.4. of \cite{bridson-book}, it is proved that the projection to a convex subset of a $\text{CAT}(0)$ space is distance non-increasing. 
This result may be improved for a $\cat(-1)$ space, as seen in Proposition \ref{proj-close}. (At least when the distance between the points being projected and the convex set is large enough.)  
In what follows, we only use this proposition when the convex set is a geodesic in $\mtilde$.
% We can prove the following Proposition using the same idea

\begin{proposition} \label{proj-close}
    Fix $A > 0$, let $\calC$ be a convex set in $\mtilde$, and let $x, y \in \mtilde$ be such that $d(x, y) \leq A$. Then
    $$d(P_\calC(x), P_\calC(y)) \prec_A e^{-d(x, \calC)}\, d(x, y)$$
\end{proposition}

\begin{proof}
    Set $h_x \coloneqq P_\calC(x)$ and $h_y \coloneqq P_\calC(y)$. If $h_x = h_y$ there is nothing to prove, otherwise let $\Delta(\barx, \barh_x, \barh_y)$ and $\Delta(\bary, \barx, \barh_y)$ be comparison triangles in $\HH^2$ for $\Delta(x, h_x, h_y)$ and $\Delta(y, x, h_y)$, such that $\barh_x$ and $\bary$ are on different sides of $[\barx, \barh_y]$. By comparison, we see that $\angle (\barx, \barh_x, \barh_y)$ and $\angle(\bary, \barh_y, \barh_x)$ are both at least $\pi/2$. Let $\bar{G}$ be the geodesic in $\HH^2$ that passes through $\barh_x$ and $\barh_y$, and let $h_\barx$ and $h_\bary$ denote the foots of perpendiculars from $\barx$ and $\bary$ to $\bar{G}$ respectively. Let $x_1$ and $y_1$ denote the reflections of $\barx$ and $\bary$ across $\bar{G}$, and let $z$ denote the point that $[x_1, \bary]$ intersects $[h_\barx, h_\bary]$. Then we have
    \begin{align*}
        |d(\bary, z) - d(\bary, h_\bary)| \leq d(z, h_\bary) \leq d(h_\barx, h_\bary) \leq d(\barx, \bary) \leq A,
    \end{align*} 
    where the second to last inequality holds since $\HH^2$ is a $\text{CAT}(0)$ space. Since $|d(\bary, h_\bary) - d(\barx, h_\barx)| \leq d(\barx, \bary)$, the above implies $d(\bary, z) = T + \bfO_A(1)$ for $T \coloneqq d(\barx, h_\barx) = d(x, C)$. By Part (\ref{converge coarse}) of Lemma \ref{same endpoint converge}, there exists a point $z' \in [\barx, x_1]$ with $d(z, z') \prec_A e^{-T} d(\barx, \bary)$. Since $h_\barx$ is the foot of projection from $z$ to $[\barx, x_1]$, we have 
    \begin{align*}
        d(z, h_\barx) \leq d(z, z') \prec_A e^{-T} d(x, y)
    \end{align*}
    In a similar way, we obtain $d(z, h_\bary) \prec_A e^{-T} d(x, y)$, thus $d(h_\barx, h_\bary) \prec_A e^{-T} d(x, y)$.
\end{proof}

Recall that $-\clow^2$ is the lower bound on the curvature of $\mtilde$. The following lemma is only used in the proof of Lemma \ref{d-zero-infty}, in which we introduce a useful semi-distance on $T^1 \mtilde$ that is \holder{}-equivalent to $d_S$.

\begin{lemma} \label{lowbound-converge}
    Fix $A > 0$, and let $x, y, z \in \mtilde$ be such that $d(z, x) = d(z, y) \geq 1$ and $d(x, y) \leq A$. Let $x_1 \in [z, x]$ and $y_1 \in [z, y]$ be such that $d(z, x_1) = d(z, y_1) = 1$. Then we have 
    $$ d(x_1, y_1) \succ_{\clow, A} e^{-\clow d(z, x)} d(x, y).$$
\end{lemma}

\begin{proof}
    By comparison, we may assume that $\mtilde$ is equal to $\HH^2_\clow$, the hyperbolic plane with constant curvature $- \clow^2$. 
    % Denoting the upper half plane by $(\HH^2, d_{\HH})$, we can obtain this plane as $(\HH, \frac{1}{\clow} d_{\HH})$. 
    Let $h$ be the foot of perpendicular from $z$ to $[x, y]$, and let $h_1$ be the intersection of this perpendicular with $[x_1, y_1]$. Denoting the angle between $[z, h]$ and $[z, y]$ at $z$ by $\alpha$, the hyperbolic sine formula in $\HH^2_\clow$ (see the second item of Theorem 7.11.2 of \cite{beardon} for $\clow = 1$), applied to triangles $\Delta(z, h, y)$ and $\Delta(z, h_1, y_1)$ gives
    \begin{align*}
        \sin \alpha = \frac{\sinh \clow \frac{d(x, y)}{2}}{\sinh \clow T}, \text{ and } \sin \alpha = \frac{\sinh \clow \frac{d(x_1, y_1)}{2}}{\sinh \clow},
    \end{align*}
    where $T \coloneqq d(z, y)$. Note that, fixing $B > 0$, we have $\sinh x \asymp_B x$ for $x \in [0, B]$ and $\sinh x \asymp_B e^x$ for $x \in [B, +\infty)$. Applying these to the above equalities, we obtain
    $$ d(x_1, y_1) \asymp_{\clow, A} e^{-\clow T} d(x, y).$$
\end{proof}

\begin{lemma} \label{d-zero-infty}
    The semi-distance $d_{0, +\infty}$ on $T^1 \mtilde$, defined by
    $$d_{0, +\infty} (u, v) \coloneqq \max\{d(\pi(u), \pi(v)),\, d_o(u^+, v^+)  \}$$
    is \holder{}-equivalent to Sassaki distance.
\end{lemma}

\begin{proof}
    We show that $d_{0, +\infty}$ is \holder{}-equivalent to $d_{0, 1}$, the Lemma follows since Sassaki metric is Lipschitz-equivalent to $d_{0, 1}$.
    Fix $R > 0$, let $u, v \in \pi^{-1} (B_R(o))$, and set
    $$x_0 \coloneqq \pi(u),\, x_1 \coloneqq \pi(\calG_1(u)),\, y_0 \coloneqq \pi(v),\, y_1 \coloneqq \pi(\calG_1(v)),\, \zeta \coloneqq u^+,\, \eta \coloneqq v^+.$$ 
    Let $w \in S(y_0)$ be such that $w^+ = \zeta$ and let $y'_1 \coloneqq \pi(\calG_1(w))$. Then a straightforward argument gives 
    $$d(x_1, y'_1) \asymp_{\clow} d(x_0, y_0).$$
    Let $T > 0$ be such that $d_{y_0}(\zeta, \eta) = e^{-T}$, thus $d(\pi(w_T), \pi(v_T)) \asymp_{\clow} 1$. By Lemma \ref{lowbound-converge} and Part (\ref{all in mtilde}) of Lemma \ref{same endpoint converge} we have
    $$e^{-\clow T} \prec d(y_1, y'_1) \prec e^{-T}.$$
    By triangle inequality we have
    $$d(x_1, y_1) \leq d(x_1, y'_1) + d(y'_1, y_1) \prec d(x_0, y_0) + e^{-T} = d(x_0, y_0) + d_{y_0}(\zeta, \eta).$$ 
    Since $d_o(\zeta, \eta) \asymp_R d_{y_0}(\zeta, \eta)$, the above implies $d(x_1, y_1) \prec_R d_{0, +\infty} (u, v)$, hence $d_{0, 1}(u, v) \prec_R d_{0, + \infty} (u, v)$.

    On the other hand, 
    $$e^{-\clow T} \prec d(y_1, y'_1) \leq d(y_1, x_1) + d(x_1, y'_1) \prec_\clow d(y_1, x_1) + d(x_0, y_0),$$
    thus
    $$d_o(\zeta, \eta) \prec_R d_{y_0}(\zeta, \eta) = e^{-T} \prec_\clow (d(y_1, x_1) + d(x_0, y_0))^{1/\clow}.$$
    This gives $d_o(\zeta, \eta) \prec_{R} (d_{0, 1}(u, v))^{1/\clow}$, hence $d_{0, +\infty} (u, v) \prec_{R} (d_{0, 1}(u, v))^{1/\clow}$.
\end{proof}

The following corollary is immediate.

\begin{corollary} \label{plus-holder}
    The map from $T^1 \mtilde$ to $\bdry$ given by $u \mapsto u^+$ is \holder{}-continuous.
\end{corollary}

\begin{lemma} \label{close endpoints converge middle}
    Fix $A, B > 0$ and let $x, x', y, y' \in \mtilde$ be such that $d(x, x')$ and $ d(y, y')$ are both at most $A$. Let $z \in [x, y]$ be such that the distance between $z$ and the midpoint of $[x, y]$ is at most $B$.
    Then there exists a point $z' \in [x', y']$ such that 
    $$d(z, z') = \bfO_{A, B} (e^{-d(x, y)/2}).$$
\end{lemma}

\begin{proof}
    Let $T \coloneqq d(x, y)$, and, without loss of generality, assume $T$ is large enough.
    Since $[x, y]$ and $[x, y']$ get exponentially close to each other (Lemma \ref{same endpoint converge}), there exists $z_1 \in [x, y']$ with 
    $$d(z, z_1) = \bfO_{A, B} (e^{-d(x, y)/2}).$$
    % for some $c_1 = c_1(a+b)$. 
    Similarly, since $[y', x]$ and $[y', x']$ get exponentially close to each other, there exists $z' \in [y', x']$ with 
    $$d(z_1, z') = \bfO (e^{-d(x, y)/2}).$$
    % for some $c_2 = c_2(a+b)$.
    % Taking $c = c_1 + c_2$ concludes the proof.
\end{proof}

The following lemma plays a major role in the proofs of Lemma \ref{step-one} and Lemma \ref{step-two}.

\begin{lemma} \label{distance approximates Busemann}
    Fix $R>0$, and let $x,\, y,\, z$, and $\zeta$ be such that $d(x, y) \leq R$ and $z \in [x, \zeta)$. Then
    % For every $R>0$, there exists $c = c(R)$ such for all $d(x, y) \leq R;\, \zeta \in \bdry;\, T>0$ and $z \in [x, \zeta)$ such that $d(z, x) = T$, we have 
    $$\beta(\zeta, x, y) = d(z, x) - d(z, y) + \bfO_R(e^{-d(x, z)}).$$
\end{lemma}

\begin{proof}
    Without loss of generality assume $\beta(\zeta, y, x) > 0$. If $y' \in [y, \zeta)$ is such that $d(y, y') = \beta(\zeta, y, x)$, then $\beta(\zeta, x, y') = 0$. We have
    $$d(x, y') \leq d(x, y) + d(y, y') \leq R + \beta(\zeta, y, x) \leq R + d(x, y) \leq 2R.$$
    Let $T \coloneqq d(x, z)$, and choose $y'_T \in [y', \zeta)$ such that $d(y', y'_T) = T$. Lemma \ref{same endpoint converge} implies $d(z, y'_T) = \bfO_R (e^{-T})$, so we have
    \begin{align*}
        d(z, x) - d(z, y) = d(z, x) - d(y'_T, y) + \bfO_R({e^{-T}})
        % d(y'_T, y) &= d(y'_T, y') + d(y', y) \\
        = \beta(\zeta, y, x) + \bfO_R(e^{-T}).
    \end{align*}
\end{proof}

\begin{lemma} \label{buseman lip-holder}
    The Busemann cocycle $\beta(\param, \param, \param)$ is \hold{\frac{1}{2}} in its first variable and \lip{1} in its second and third variables.
    % and for every $R >0$, there exists $c = c(R)$ such that if $x, y \in B_R(o)$, $\beta$ is (with multiplicative constant $c$) in its first variable.
\end{lemma}

\begin{proof}
    For $\zeta \in \bdry$ and $x,\, x',\, y \in \mtilde$, we have
    % and let $v \in T_o \mtilde$ be tangent to $[o, \zeta)$ at $o$. Then
    % \begin{align*}
    %     \beta(\zeta, x', y) &= \lim_{t \to \infty} d(v_t, x') - d(v_t, y);\\
    %     \beta(\zeta, x, y) &= \lim_{t \to \infty} d(v_t, x) - d(v_t, y),
    % \end{align*}
    % hence
    \begin{align*}
        |\beta(\zeta, x', y) - \beta(\zeta, x, y)| = |\beta(\zeta, x', x)| \leq d(x, x').
        % \lim_{t \to \infty} |d(v_t, x') - d(v_t, x) | \leq d(x, x').
    \end{align*}
    This proves that $\beta$ is \lip{1} in its second variable. In a similar way, $\beta$ is \lip{1} in its third variable.
    
    To prove \holder{}-continuity in the first variable,
    let $R > 0$ be arbitrary, and fix $x, y \in \calB_R(o)$.
    Let $\zeta, \eta \in \bdry$ and let $T$ be such that $d_x(\zeta, \eta) = e^{-T}$. Choose $x_\zeta, x'_\zeta \in [x, \zeta)$ and $ x_\eta, x'_\eta \in [x, \eta)$ such that $d(x, x_\zeta) = d(x, x_\eta) = T$, and $d(x, x'_\zeta) = d(x, x'_\eta) = T/2$. Since $d(x_\zeta, x_\eta) \asymp 1$, Lemma \ref{same endpoint converge} implies 
    $d(x'_\zeta, x'_\eta)= \bfO(e^{-T/2})$. By Lemma \ref{distance approximates Busemann} we have
    \begin{align*}
        \beta(\zeta, x, y) &= d(x'_\zeta, x) - d(x'_\zeta, y) + \bfO_R(e^{-T/2});\\
        \beta(\eta, x, y) &= d(x'_\eta, x) - d(x'_\eta, y) + \bfO_R(e^{-T/2}).
    \end{align*} 
    Thus,
    \begin{align*}
        |\beta(\zeta, x, y) - \beta(\eta ,x , y)| & \leq 
        |d(x'_\zeta, x) - d(x'_\eta, x)| + |d(x'_\zeta, y) - d(x'_\eta, y) | + \bfO_R({e^{-T/2}}) \\
        & \leq 2 d(x'_\zeta, x'_\eta) + \bfO_R({e^{-T/2}})= \bfO_R({e^{-T/2}}) 
        % = \bfO_R (d_x(\zeta, \eta)^{\frac{1}{2}}) 
        = \bfO_R (d_o(\zeta, \eta)^{\frac{1}{2}}),
    \end{align*}
    where to obtain the last equality we used the fact that $d_x(\zeta, \eta) \asymp_R d_o(\zeta, \eta)$.
\end{proof}

The following lemma is only used in the proof of Lemma \ref{proj. is holder}, in which we prove the \holder{}-continuity of a function that appears in the proof of Lemma \ref{step-one}.

\begin{lemma} \label{perp. increase length}
    Fix $A > 0$, and let $G$ be a geodesic in $\mtilde$ and $x$ a point in $\mtilde$ such that $d(x, G) \leq A$. Set $h \coloneqq P_G(x)$ and assume $h' \in G$ is such that $d(h, h') \leq 1$. Then we have 
    $$d(x, h') - d(x, h) \succ_A d(h, h')^2.$$
\end{lemma}

\begin{proof}
    Let $\Delta(\bar{x}, \bar{h}, \bar{h}')$ be a triangle in the Euclidean plane such that $d(\barx, \barh) = d(x, h)$, $d(\barh', \barh) = d(h', h)$, and the angle that $[\barh, \barx]$ and $[\barh, \barh']$ make at $\barh$ is equal to $\frac{\pi}{2}$. By comparison, it is enough to prove the above inequality for $\Delta(\bar{x}, \bar{h}, \bar{h}')$, which is an exercise in Euclidean geometry.
\end{proof}

\begin{lemma} \label{proj. is holder}
    The function $P_o \from \bdry \times \bdry \to \mtilde$ defined by $P_o(\zeta_1, \zeta_2) = P_{(\zeta_1, \zeta_2)} (o)$ is \hold{\frac{1}{2}} continuous.
\end{lemma}

\begin{proof}
    For $R > 0$, define the compact set 
    $\mathbf{K}_R \coloneqq \{ (\zeta_1, \zeta_2): d(o, (\zeta_1, \zeta_2)) \leq R) \}$, and let $\bfzeta \coloneqq (\zeta_1, \zeta_2)$ and $\bfeta \coloneqq  (\eta_1, \eta_2)$ be elements in $\mathbf{K}_R$. Define $T > 0$ to be such that 
    $$e^{-T} = d( \bfzeta, \bfeta) = \max \{d_o(\zeta_1, \eta_1),\, d_o(\zeta_2, \eta_2)  \}. $$
    Without loss of generality we can assume that $T$ is large enough.
    By the definition of $d_o$, if $x_1$ and $y_1$ are points at distance $T$ from $o$ on $[o, \zeta_1)$ and $[o, \eta_1)$ respectively, we have $d(x_1, y_1) \prec 1$. 
    Letting $h_{\bfzeta} \coloneqq P_o(\bfzeta)$, since $d(o, h_{\bfzeta}) \leq R$ and $T$ is large enough, Lemma \ref{same endpoint converge} Part (\ref{one in boundary coarse}) gives $x'_1 \in [h_{\bfzeta}, \zeta_1)$ with $d(x_1, x'_1) \leq 1$. Letting $h_{\bfeta} \coloneqq P_o(\bfeta)$, a similar argument gives $y'_1 \in [h_{\bfeta}, \eta_1)$ with $d(y_1, y'_1) \leq 1$. Triangle inequality then implies
    $$ d(x'_1, y'_1) \leq d(x'_1, x_1) + d(x_1, y_1) + d(y_1, y'_1) \prec 1.$$
    The same argument for $\zeta_2, \eta_2$ replacing $\zeta_1, \eta_1$ gives $x'_2 \in [h_{\bfzeta}, \zeta_2)$ and $y'_2 \in [h_{\bfeta}, \eta_2)$ with $ d(x'_2, y'_2) \prec 1$.
    Since $d(x'_1, h_{\bfzeta})$ and $d( h_{\bfzeta}, x'_2)$ are both $T + \bfO({R})$, 
    Lemma \ref{close endpoints converge middle} applied to the quadruple $(x'_1, y'_1, x'_2, y'_2)$ gives a point $h'_{\bfzeta} \in (\eta_1, \eta_2)$ with $d(h_{\bfzeta}, h'_{\bfzeta}) = \bfO(e^{-T})$.
    Hence we have
    \begin{align*} 
        % \begin{split}
        d(o, h_{\bfeta}) = \inf \{ d(o, h) : h \in (\eta_1, \eta_2) \} \leq d(o, h'_{\bfzeta})  
        \leq d(o, h_{\bfzeta}) + d(h_{\bfzeta}, h'_{\bfzeta}) \leq d(o, h_{\bfzeta}) + c e^{-T}. 
        % \end{split}   
    \end{align*}
    for some constant $c$.
    In a similar fashion, Lemma \ref{close endpoints converge middle} gives a point $h'_{\bfeta} \in (\zeta_1, \zeta_2)$ with $d(h_{\bfeta}, h'_{\bfeta}) = \bfO(e^{-T})$, thus we have
    \begin{align*} 
        d(o, h'_{\bfeta}) \leq d(o, h_{\bfeta}) + d(h_{\bfeta}, h'_{\bfeta}) \leq d(o, h_{\bfeta}) + c' e^{-T} \leq d(o, h_{\bfzeta}) + c'' e^{-T},
    \end{align*}
    for some constants $c'$ and $c''$.
    Since $T$ is large enough, the above gives $d(o, h'_{\bfeta}) - d(o, h_{\bfzeta}) < 1$. Thus by Lemma \ref{perp. increase length}, 
    $$ d(h_{\bfzeta}, h'_{\bfeta})^2 \prec_R d(o, h'_{\bfeta}) - d(o, h_{\bfzeta}),$$
    which implies $d(h_{\bfzeta}, h'_{\bfeta}) =\bfO_R (e^{-T/2})$. By triangle inequality
    $$ d(h_{\bfzeta}, h_{\bfeta}) \leq d(h_{\bfzeta}, h'_{\bfeta}) + d(h'_{\bfeta}, h_{\bfeta}) = \bfO(e^{-\frac{T}{2}}).$$
    This concludes the proof.
\end{proof}

\begin{lemma} \label{passes near midpoint}
    Let $\gamma \in \Gamma$ and denote the axis of $\gamma$ by $L$. Then there exists a constant $c = c(\gamma)$ such that for every 
    $\zeta \in \bdry \setminus \fix(\gamma)$, the geodesic $(\zeta, \gamma \zeta)$ passes through the ball of radius $c$ centered at the midpoint of $[P_L (\zeta), \gamma. P_L (\zeta) ]$.
    % if $\gamma \in \Gamma$ is such that the translation length of $\gamma$ is greater than $\ell$ and 
    % For every $\ell > 0$, there exists a constant $c = c(\ell)$ such that the following holds: 
    % where $L_\gamma$ denotes the axis of $\gamma$.
\end{lemma}

\begin{proof}
    Define $f \from \bdry \to \RR$ by setting $f(\zeta)$ to be the distance between the midpoint of $[P_L (\zeta), \gamma. P_L (\zeta) ]$ and the geodesic $(\zeta, \gamma \zeta)$. Fixing $x_0 \in L$, since $f$ is continuous, it attains its maximum $c$ on the compact set $P_L^{-1} [x_0, \gamma x_0]$. Since $f$ is $\gamma$--invariant, $c$ is the desired constant. 
    % function is continuous and $\gamma$--invariant, 
    % Fix $x_0 \in L$ and let $c$ be the maximum of $f$ on the compact set $P_L^{-1} [x_0, \gamma x_0]$. This is the desired constant.
\end{proof}

\section{Reduction to adjusted counts} \label{reduction-control}

The goal of this section is to prove Proposition \ref{reduction-prop}, which is directly used in the proof of Theorem \ref{main-theorem}, given right after Theorem \ref{control counting}.
% We fix an element $\gamma \in \Gamma$ for the rest of this section, and denote its axis by $L \coloneqq L_\gamma$, and assume the origin $o$ lies on $L$. Henceforth,  
% Fix points $x,  y, o$.

\begin{proposition} \label{reduction-prop}
    Let $x, y \in \mtilde,\, \id \neq \gamma \in \Gamma$, and denote the axis of $\gamma$ by $L$. Then there are \holder{}-continuous functions $F_1 \from \partial^1 L \to \RR$ and $F_2 \from S(x) \to \RR$ such that $F_1$ is $\gamma$--invariant, and for every $g \in \Gamma$ we have
    $$ d(gx, \gamma. g y) = 2 d(gx, L) + F_1 (v_1(g)) + F_2(g^{-1}. v_2(g)) + \bfO_{x, y} (e^{- d(gx, L)/2 }),$$
    where $v_1(g)$ and $v_2(g)$ are tangents to $[gx, P_L(gx)]$ at $P_L(gx)$ and $gx$ respectively.
    % for $v_1(g) \coloneqq P^1_L(gx)$ and $v_2(g) \coloneqq P_{gx}(L)$.
\end{proposition}

\begin{proof}
    The proposition follows from Lemma \ref{step-one} and Lemma \ref{step-two}, proved below.
\end{proof}

We fix $x,\, y,\, \gamma,\, L$ to be as in Proposition \ref{reduction-prop}, and for the rest of this section we do not show the dependence of the implied constants in $\bfO_\param(\param)$ on these fixed parameters; for example, instead of $A = \bfO_\gamma (B)$ we write $A = \bfO(B)$.

\begin{lemma} \label{step-one}
    Define the function $F_1 \from \partial^1 L \to \RR$ by 
    \begin{align} \label{def-fone}
        F_1(v) \coloneqq \beta(v^+, P_{(v^+, \gamma v^+)} o, \pi(v)\,) + \beta(\gamma v^+, P_{(v^+, \gamma v^+)} o, \gamma \pi(v)\,).
    \end{align}
    Then $F_1$ is $\gamma$--invariant and \holder{}-continuous, and for every $g \in \Gamma$,
    \begin{align} \label{step 1 formula}
        d(gx, \gamma gx) - 2 d(gx, L) = F_1(v_1(g)) + \bfO({e^{-d(gx, L)}}).
    \end{align}
\end{lemma}

\begin{proof}
    % \emph{Step 1.} 
    % for some $\usec{c3}$ that only depends on $\ell$. ($v_1(g)$ is as defined in the statement of the proposition.) 
    The map $v \mapsto v^+$ is \holder{}-continuous by Corollary \ref{plus-holder}, $P_{(\param, \param)} o$ and $\beta(\param, \param, \param)$ are \holder{}-continuous by Lemma \ref{proj. is holder} and Lemma \ref{buseman lip-holder}, and $\pi(\param)$ is \lip{1} (see Section 2.3. of \cite{pps-book}). One can check that $\gamma \from \partial_\infty \mtilde \to \bdry$ is also \holder{}-continuous (in fact, it is Lipschitz), thus $F_1$, being a combination of \holder{}-continuous functions, is \holder{}-continuous itself. 
    % Note that the right hand side of (\ref{def-fone}) remains the same if $P_{(v^+, \gamma v^+)} o$ is replaced by any other point $z \in (v^+, \gamma v^+)$, hence
    % \begin{align*}
    %     F_1(\gamma v) = \beta(\gamma v^+, P_{(\gamma v^+, \gamma^2 v^+)} o, \pi(\gamma v)\,) + \beta(\gamma^2 v^+, \gamma. \pi(\gamma v), P_{(\gamma v^+, \gamma^2 v^+)} o\,)\\ 
    %     =  \beta(\gamma v^+, \gamma. P_{(v^+, \gamma v^+)} o, \gamma. \pi(v)) + \beta(\gamma^2 v^+, \gamma^2. \pi(v), \gamma. P_{(v^+, \gamma v^+)} o) \\
    %     = \beta(v^+, P_{(v^+, \gamma v^+)} o, \pi(v)\,) + \beta(\gamma v^+, \gamma. \pi(v), P_{(v^+, \gamma v^+)} o\,) = F_1(v),
    % \end{align*}
    % where we used the fact that Busemann function is invariant under isometries to obtain the second to last equality.
    % This proves that $F_1$ is $\gamma$--invariant.  
    If we change $g$ with $\gamma^i g$ for some $i \in \ZZ$, all the terms on the left and right hand sides of (\ref{step 1 formula}) remain the same, hence, without loss of generality, we can assume $d(o, P_L(gx)) = \bfO(1)$. We can moreover assume that $d(gx, L)$ is large enough. We fix such an element $g$ for the rest of the proof. 
    
    Let $x_1 \coloneqq gx,\, x_2 \coloneqq \gamma gx,\, h \coloneqq P_L(x_1)$, and $\zeta \coloneqq (v_1(g))^+$. 
    Letting $m$ denote the midpoint of $[h, \gamma h]$, by Lemma \ref{passes near midpoint} there is a point $z \in (\zeta, \gamma \zeta)$ with $d(z, m) = \bfO(1)$, thus, by triangle inequality, $d(h, z) = \bfO(1)$.
    Since $d(gx, L)$ is large enough and $(\zeta, z]$ and $(\zeta, h]$ converge (see Lemma \ref{same endpoint converge}), we can find $x_1' \in (\zeta, z]$ with $d(x_1, x_1') < 1$. 
    The same argument applied to $(\gamma \zeta, z]$ and $(\gamma \zeta, \gamma h]$ implies the existance of $x_2' \in (\gamma \zeta, z]$ with $d(x_2, x_2') < 1$. By triangle inequality we have
    \begin{align*}
        |d(x'_1, z) - d(x_1, h)| &\leq d(x'_1, x_1) + d(h, z) = \bfO(1),
        % & \implies d(x'_1, z) = T + \bfO({\ell + }).
    \end{align*}
    thus $d(x'_1, z) = d(x_1, h) + \bfO(1)$.
    Similarly, we obtain $d(x'_2, z) = d(x_2, \gamma h) + \bfO(1) = d(x_1, h) + \bfO(1)$, hence $d(x'_1, z) = d(x'_2, z) + \bfO(1)$.
    This means that $z$ is close to the midpoint of $[x'_1, x'_2]$.
    Since $[x_1, x_2]$ and $[x_1', x_2']$ have close endpoints, by Lemma \ref{close endpoints converge middle} they get exponentially close in the middle. More precisely, there exists $z' \in [x_1, x_2]$ with
    % \nextc{c2}
    $d(z, z') = \bfO (e^{-d(gx, L)})$. We have
    \begin{align*} 
        d(gx, \gamma gx) - 2 d(gx, L) & = d(x_1, x_2) - d(x_1, h) - d(x_2, \gamma h) \\
        & = d(x_1, z') + d(z', x_2) - d(x_1, h) - d(x_2, \gamma h)\\
        &= \big( d(x_1, z') - d(x_1, h) \big) + \big(  d(x_2, z') - d(x_2, \gamma h) \big)\\
        &= \big( d(x_1, z) - d(x_1, h) \big) + \big(  d(x_2, z) - d(x_2, \gamma h) \big) + \bfO({ e^{-d(gx, L)}  })\\
        & = \beta(\zeta, z, h) + \beta(\gamma \zeta, z, \gamma h) +
        \mathbf{O}(e^{-d(gx, L)}),
        %\tag*{(by Lemma \ref{distance approximates Busemann})},
        % & = \beta(\zeta, z, h) + \beta(\gamma \zeta, \gamma h, z) + \bfO{(\usec{c2} + + \ell) e^{-T}}. \tag*{(by Lemma \ref{buseman lip-holder})}
    \end{align*}
    where the last equality is by Lemma \ref{distance approximates Busemann}.
    Since $\beta(\zeta, z, h) + \beta(\gamma \zeta, \gamma h, z)$ remains the same if $z$ is replaced by an arbitrary $z''  \in (\zeta, \gamma \zeta)$, (\ref{step 1 formula}) follows.
    
    For $t > 0$, we can repeat the above argument for $x_1$ and $x_2$ replaced by $\pi (v_t)$ and $g. \pi(v_t)$ respectively, to obtain
    \begin{align*}
        d(\pi (v_t),  g. \pi(v_t)) - 2t = F_1(v) + \bfO(e^{-t}),
    \end{align*}
    thus
    \begin{align} \label{F_in_lim}
        F_1(v) = \lim_{t \to \infty} d(\pi (v_t),  g. \pi(v_t)) - 2t.
    \end{align}
    It is clear from the above formula that $F_1$ is $\gamma$--invariant.
\end{proof}

\begin{lemma} \label{step-two}
    Define the function $F_2 \from S(x) \to \RR$ by 
    \begin{align*}
        F_2(v) \coloneqq \beta(v^+, y, x).
    \end{align*}
    Then $F_2$ is \holder{}-continuous, and for every $g \in \Gamma$,
    \begin{align} \label{step 2 formula 0}
        d(gx, \gamma gy) - d(gx, \gamma gx) = F_2(g^{-1}. v_2(g)) + \bfO(e^{- d(gx, L) /2}).
    \end{align}
\end{lemma}

\begin{proof}
    The map $v \mapsto v^+$ is \holder{}-continuous by Corollary \ref{plus-holder} and $\beta(\param, \param, \param)$ is \holder{}-continuous by Lemma \ref{buseman lip-holder}, thus $F_2$ is \holder{}-continuous.
    Fix $g \in \Gamma$ and let $x_1,\, x_2,\, h,\, z'$ be as in Lemma \ref{step-one}.
    %Assume $z_1$ is a point in $[x_1, x_2]$ with $d(z_1, m) < \usec{c1}$ and  
    Letting $\bar{x}_1 \coloneqq (\gamma g)^{-1} x_1$, we have 
    \begin{align*}
        d(gx, \gamma gy) - d(gx, \gamma gx) = d(x_1, \gamma gy) - d(x_1, x_2)
            = d(\bar{x}_1, y ) - d(\bar{x}_1, x).
        %& = \beta(w_1^+, y, x) + \bfO{(\usec{c1} + \ell) e^{-2T}}. \nonumber
    \end{align*}
    Set $\bar{h} \coloneqq (\gamma g)^{-1} \gamma h$ and $\bar{z}' \coloneqq (\gamma g)^{-1} z'$, let $v$ and $w$ be tangents to $[x, \bar{h}]$ and $[x, \bar{z}']$ at $x$ respectively, and note that $v = g^{-1}.v_2(g)$.
    % The last equality follows from Lemma \ref{distance approximates Busemann}, since 
    Since $d(\bar{x}_1, x) = d(x_1, x_2) = 2 d(gx, L) + \bfO(1)$, Lemma \ref{distance approximates Busemann} implies 
    \begin{align*}
        d(\bar{x}_1, y ) - d(\bar{x}_1, x) = \beta(w^+, y, x) + \bfO(e^{-2d(gx, L)}).
    \end{align*}
    Since $d(\bar{h}, \bar{z}') = d(\gamma h, z') = \bfO({1})$ and $d(x, \bar{h}) = d(gx, L)$, part (\ref{all in mtilde}) of Lemma \ref{same endpoint converge} implies $d_x(v^+, w^+) = \bfO(e^{-d(gx, L)})$. Thus, 
    by Lemma \ref{buseman lip-holder}, 
    \begin{align*}
        \beta(w^+, y, x) = \beta(v^+, y, x) + \bfO(e^{-d(gx, L)/2}),
    \end{align*}
    % \nextc{c7}
    which concludes the proof.
\end{proof}

\section{The test functions} \label{test-intro}

The goal of this section is to set the stage for Section \ref{count-control}, in which we obtain asymptotics for adjusted counts with a power saving error term. (See Theorem \ref{control counting}.)
% We adopt the notation and conventions introduced at the beginning of Section \ref{reduction-control}. 
This section and the next closely follow sections 2-4 of \cite{PPCommonPerp}, and Theorem \ref{control counting} is a generalization of Thoerem 3 of that paper (when the locally convex sets $D^\pm$ in the statement of that theorem are replaced by $\Pi(x)$ and $\Pi(L)$). Since we make frequent references to \cite{PPCommonPerp}, it is useful to summarize the main differences between our exposition and theirs.
\begin{enumerate} [(i)]
    \item \label{thicken} Instead of a \emph{thickening} (or a \emph{dynamical neighbourhood}) of $\partial^1 L$ (resp. $S(x)$), we first flow $\partial^1 L$ (resp. $S(x)$) by a \holder{}-continuous function, and then consider the dynamical thickening of this latter set. (Compare (\ref{defoftest}) and Equation (16) of \cite{PPCommonPerp}.)
    \item Instead of the \Ham distance on the strongly stable foliation, we use the distance induced by $d_{\mtilde}$. This is essential to establish the RHC property. (See Remark \ref{two-ham}.)
    \item \label{t and t two} Instead of flowing the dynamical neighbourhoods mentioned in (\ref{thicken}) by $\frac{t}{2}$, we flow the neighbourhood corresponding to $\partial^1 L$ by time $t$ and keep the neighbourhood corresponding to $S(x)$ fixed. This keeps our presentation closer to the original paper of \cite{roblin-thesis}.
\end{enumerate}

As mentioned in the introduction, our main contributions are as follows.
\begin{enumerate} [(a)]
    \item \label{rhcdone} We verify the RHC condition under the assumptions of Theorem \ref{main-theorem}. (See Section \ref{rhc-section}.) 
    % This has entered as an assumption in Theorem 3 of \cite{PPCommonPerp}.
    \item \label{smoothdone} We provide the details of a `smoothing argument' which is only sketched in \cite{PPCommonPerp, PPSkin}. (See Appendix \ref{construct-chi}.)
\end{enumerate}
Both of the above are necessary to obtain a power saving error term. In fact, (\ref{rhcdone}) is used in the proof of Lemma \ref{j almost one}, and (\ref{smoothdone}) is used in the proof of Lemma \ref{phi-mixing}.

% \subsection{Notation and definitions} \label{counting-notation}
% We adopt the notation and conventions introduced at the beginning of Section \ref{reduction-control}.
For a subset $A$ of a topological space $X$, we denote the characteristic function of $A$ by $\mathbbm{1}(A)$. For a compactly supported function $\varphi \from T^1 \mtilde \to \RR$, define the function $\bar{\varphi} \from T^1 M \to \RR$ by
\begin{align} \label{pushdown-point}
    \bar{\varphi}(u) \coloneqq \sum_{u' \in \Pi^{-1}(u)} \varphi(u'),
\end{align}
% where we denote the image of $u \in T^1 \mtilde$ under the covering map $T^1 \mtilde \to T^1 M$ by $\bar{u}$. 
and note that
\begin{align*}
    \int_{T^1 M} \bar{\varphi} = 
    \int_{T^1 \mtilde} \varphi.
\end{align*}

Fix a point $y_0 \in L$, and note that the set $\Delta_\gamma \coloneqq \{u \in T^1 \mtilde: P_L(u^+) \in [y_0, \gamma. y_0)\}$ is a fundamental domain for the action of $\gamma$ on $T^1 \mtilde \setminus \{u: u^+ \in \fix(\gamma)\}$. For a $\gamma$--invariant function $\psi \from T^1 \mtilde \to \RR$ such that $\psi \times \mathbbm{1}(\Delta_\gamma)$ is compactly supported, define the function $\bar{\psi} _\gamma \from T^1 M \to \RR$ by
\begin{align} \label{pushdown-line}
    \bar{\psi}_\gamma (u) \coloneqq 
    \sum_{u' \in \Pi^{-1} (u) \cap \Delta_\gamma} \psi(u'),
\end{align}
% Note that $\bar{\psi}_\gamma$ is independent of the choice of $y$, 
and note that 
\begin{align} \label{int-pushdown}
    \int_{T^1 M} \bar{\psi}_\gamma = 
    \int_{T^1 \mtilde} \psi \times \mathbbm{1}(\Delta_\gamma).
\end{align}

A \emph{Patterson-Sullivan density} for $\Gamma$ is a family $(\mu_x)_{x \in \mtilde}$ of finite measures such that for every $x \in \mtilde$ and $\gamma \in \Gamma$ we have $\gamma_* \mu_x = \mu_{\gamma x}$; and for all $x, y \in \mtilde$ and (almost) all $\zeta \in \bdry$, we have
$$\frac{d\mu_x}{d\mu_y}(\zeta) = e^{-\delta \beta(\zeta, x, y)}.$$
Such a family exists and is unique upto a multiplicative constant. We pick the fmaily that makes $\mbarBM$, defined just above Proposition \ref{integral-of-test}, a probability measure. To simplify the notation, from now on we set $\beta_o(\zeta, x) \coloneqq \beta(\zeta, x, o)$, for $x \in \mtilde$ and $\zeta \in \bdry$.

Fix $u \in T^1 \mtilde$ for the rest of this paragraph. Define the \emph{strongly stable manifold} of $u$ by 
\begin{align*}
    \wss{u} \coloneqq \{v \in T^1 \mtilde: v^+ = u^+, \,\text{ and }\, \beta(u^+, \pi(u), \pi(v)) = 0 \}.
\end{align*}
Using the homeomorphism $v \mapsto v^-$ from $\wss{u}$ to $\bdry - \{u^+\}$, define the measure $\muss_u$ on $\wss{u}$ by
\begin{align*}
    d \muss_u (v) \coloneqq e^{- \delta \beta_o(v^-, \pi(v))}\, d\mu_o(v^-).
\end{align*}
More precisely, denoting the pushforward of $\mu_o$ to $\wss{u}$ under the above-mentioned homeomorphism by $\mu'_o$, we have
\begin{align*}
    d \muss_u(v) = e^{- \delta \beta_o(v^-, \pi(v))}\, d\mu'_o(v).
\end{align*}
Define the \emph{stable manifold} of $u$ by $\ws(u) \coloneqq \{v \in T^1 \mtilde: v^+ = u^+ \}$. Using the homeomorphism $(v, t) \mapsto w = \calG_t(v)$ from $\wss{u} \times \RR$ to $\ws(u)$, define the measure $\mus_u$ on $\ws(u)$ by
\begin{align*}
    d \mus_u(w) \coloneqq e^{-\delta t} \,d\muss_u(v) \, dt.
\end{align*}
Note that $\muss_u$ (resp. $\mus_u$) only depends on $\wss{u}$ (resp. $\ws(u)$).
Given $r > 0$, define the strongly stable ball of radius $r$ centered at $u$ by
\begin{align*}
    \bss{u, r} \coloneqq \{ v \in \wss{u} : d_{\mtilde} (\pi(u), \pi(v)) \leq r \},
\end{align*}
and let $\mub{u, r} \coloneqq \muss_u (\bss{u, r} )$.
For $\tau >0$, define 
\begin{align*}
    \thickbs{\tau}{u, r} \coloneqq \bigcup_{t \in [-\tau, \tau]} 
    \calG_t(\bss{u, r}),
\end{align*}
and for a subset $S \subset T^1 \mtilde$, define $\thickbs{\tau}{S, r}$ to be the union of $\thickbs{\tau}{ u, r}$ for all $u \in S$.

\begin{remark} \label{two-ham}
    In \cite{PPCommonPerp}, the strongly stable balls are defined using the \Ham distance instead of $d_{\mtilde}$.
    In that paper (and in many other references, ex. \cite{pps-book, PPSkin, pulin-tree-book}) the \Ham distance between $v, v' \in \wss{u}$ is defined by  
    $$d^{\ham}_u (v, v') \coloneqq \lim_{t \to +\infty} e^{\frac{1}{2} d(\pi(v_{-t}),\, \pi(v'_{-t})) - t }.$$ 
    In the original paper of \Ham \cite{hamdist}, however, for a fixed positive real number $R$, a different distance $d_u^{\ham, R}$ is defined as follows. (This distance is denoted by $\eta_{u, R}$ in \cite{hamdist}). For $t \in \RR$, let $d_t$ denote the distance on $\wss{u_t}$ induced by restriction of the Riemannian structure of $\mtilde$ to its submanifold $\pi(\wss{u_t})$. For $v, v' \in \wss{u}$, let $T$ be such that $d_{-T}(v_{-T}, v'_{-T}) = R$, and set  
    $$d_u^{\ham, R}(v, v') \coloneqq e^{-T}.$$
    The distances $d^{\ham}$ and $d^{\ham, R}$ are the same upto a multiplicative constant that only depends on $\clow$ and $R$, and the equality $d^{\ham}_{u_t}(v_t, v'_t) = e^{-t} d^{\ham}_u(v, v')$ also holds for $d^{\ham, R}$.

    To verify the RHC condition in Section \ref{rhc-section}, we need the boundary of $\bss{u, r}$ to be $C^1$, thus we can use $d_u^{\ham, R}$ (but not $d_u^{\ham}$) instead of $d_{\mtilde}$ in the definition of $\bss{u, r}$.
    Finally, it is worth mentioning that the strongly stable balls considered in \cite{roblin-thesis} are different from both ours and \cite{PPCommonPerp}. Compare $\thickbs{\frac{r}{2}}{u, r}$ with $K^+(\pi(u), r, u^+)$ defined in Chapter 4 of \cite{roblin-thesis}.
\end{remark}

For the following definitions let $D$ be either a point $x \in \mtilde$, or equal to $L$.
% , the axis of an element $\gamma \in \Gamma$. 
In the former case define $\partial^1 D \coloneqq S(x)$, and in the latter, set $\partial^1 D$ to be the unite normal bundle of $L$. 
Let $F$ be a real-valued function on $\partial^1 D$, and define
\begin{align*}
    \calG^F \from \partial^1 D \to T^1 \mtilde \gap \text{ by } \gap
    \calG^F(u) = \calG_{F(u)} (u).
\end{align*}
For $u \in T^1 \mtilde$, set $p_D(u) \coloneqq P^1_D(u^+)$, and for $\tau, r>0$ define the test function $\test{D}{F}{r, \tau} \from T^1 \mtilde \to \RR$ by
\begin{align} \label{defoftest}
    \test{D}{F}{r, \tau} (u) \coloneqq \frac{1}{
        2 \tau \mub{\calG^F(p_D u), r}
        } \times
    \mathbbm{1} \big( \thickbs{\tau}{ \calG^F(\partial^1 D), r} \big) (u).
\end{align}
$\thickbs{\tau}{ \calG^F(\partial^1 D), r}$ is called a thickening, or a dynamical neighbourhood of $\calG^F(\partial^1 D)$.
% where for a subset $S \subset \mtilde$, we denote the characteristic function of $S$ by $\mathbbm{1}(S)$. 
By using (\ref{pushdown-point}) or (\ref{pushdown-line}), depending on whether $D$ is a point or is equal to $L$, $\test{D}{F}{r, \tau}$ descends to a function from $T^1 M$ to $\RR$, which we denote by $\bartest{D}{F}{r, \tau}$.

We define the \emph{skinning measure} $\sigma_D$ on $\partial^1 D$ by
\begin{align} \label{skin def}
    d \sigma_D(v) \coloneqq e^{-\delta \beta_o(v^+, \pi(v)) } d\mu_o(v^+).
\end{align}
If $D = x$ is a point, then $d\sigma_D(v) = d \mu_x(v^+)$, and for $D = L$, $\sigma_D$ is $\gamma$--invariant (\cite{PPSkin} Proposition 4(ii)). For a real valued function $F_1$ defined on $\partial^1 L$, define  
\begin{align} \label{sigmafdef_1}
    \sigma_\gamma (F_1) \coloneqq \int_{\partial^1[y_0, \gamma. y_0)} e^{\delta F_1(u)}\, d  \sigma_{L}(u) \quad \text{ for some } \quad y_0 \in L,
\end{align}
and for a real valued function $F_2$ defined on $S(x)$, set 
\begin{align} \label{sigmafdef_2}
    \sigma_x(F_2) \coloneqq \int_{S(x)} e^{\delta F_2(u)}\, d  \sigma_x(u).
\end{align}

%change of coordinates
For $u \in T^1 \mtilde$, let $\zeta \coloneqq u^+$, $\eta \coloneqq u^-$, and $t \coloneqq d\big( P_{(\eta, \zeta)} o, \pi(u)\, \big)$. The map $u \mapsto (\zeta, \eta, t)$ gives a homeomorphism between $T^1 \mtilde$ and $\bdry \times \bdry \setminus \{(\zeta, \zeta): \zeta \in \bdry\}$. This is called the \emph{Hopf parametrization}. Note that the parameter $t$ above is the \emph{signed distance} between $P_{(\eta, \zeta)} o$ and $\pi(u)$, and the sign is chosen in a way that $\calG_T$ in Hopf coordinates is given by $(\zeta, \eta, t) \mapsto (\zeta, \eta, t+T)$.
% \begin{align} \label{hopf-eq}
%     u \mapsto (\zeta, \eta, t) = (u^+, u^-, d( P_{(u^-, u^+)} o, \pi(u) ) ) 
%     % u \mapsto (v, w, t) \text{ such that } v = p_D(u)
% \end{align}
% where $d(\param, \param)$ is the signed distance, i.e., it is defined in a way that $\calG_T$ is given by 
% \begin{align*}
%     (\zeta, \eta, t) \mapsto (\zeta, \eta, t+T)
% \end{align*}
% in these coordinates. 
We can write the \emph{Bowen-Margulis measure} in these coordinates as 
\begin{align*}
    d\mBM = e^{-\delta d_0(\zeta, \eta)  }\, d \mu_{o} (\zeta)\, d \mu_{o} (\eta) \, dt,
\end{align*}
where 
$d_o(\zeta, \eta) \coloneqq \beta_o(\zeta, y) + \beta_o(\eta, y)$ for some (and hence all) $y \in (\zeta, \eta)$. 
Since $\mBM$ is $\Gamma$--invariant, it descends to a measure on $T^1 M$, which we denote by $\mbarBM$. 

% For a given $u \in T^1 \mtilde$, define $v_D = p_D(u)$, and let $w_D \in \wss{v_D}, t_D \in \RR$ be such that $\calG_{t_D}(w_D) = u$. Then we call
% %(fibration or coordinates)!??
% \begin{align*}
%     u \mapsto (v_D, w_D, t_D)
% \end{align*}
% the Hopf fibration associated to $D$.
%$\bartest{D}{F}{r}{\tau}$

\begin{proposition} \label{integral-of-test}
    Let $D$ be either a point $x \in \mtilde$, or equal to $L$, and let $F$ be a real valued function on $\partial^1 D$. Then we have
    \begin{align*}
        \int_{T^1 M} 
        \bartest{D}{F}{r, \tau} (u) \, d \mbarBM(u) = \sigma_\bullet(F),
    \end{align*}
    where $\bullet = \gamma$ (resp. $\bullet = x$) when $D = L$ (resp. $D = x$).
\end{proposition}

\begin{proof}
    We prove the proposition for $D = L$. The proof is similar for the case where $D$ is a point.
    Let $U_L \coloneqq \{u \in T^1 \mtilde: u^+ \notin \lbdry \}$. Then $p_L \from U_L \to \partial^1 L$ is a fibration, and by Proposition 8 of \cite{PPSkin}, under this fibration, the Bowen-Margulis measure on $U_L$ disintegrates over the skinning measure $\sigma_L$ on $\partial^1 L$, with the conditional measures 
    $e^{\delta \beta(v^+, \pi(v), \pi(w))}\, d \mus_v(w)$ on $p_L^{-1}(v) = \ws(v)$ for $v \in \partial^1 L$.

    Fix a point $y_0 \in L$, and note that $\Delta_\gamma \coloneqq \{u \in T^1 \mtilde: P_L(u^+) \in [y_0, \gamma. y_0)\}$ is a fundamental domain for the action of $\gamma$ on $U_L$. By (\ref{int-pushdown}) we have
    \begin{align} \label{interated-int}
        \int_{T^1 M} \bartest{L}{F}{r, \tau} (u) \,d \mbarBM(u) &= 
        \int_{\Delta_\gamma} \test{L}{F}{r, \tau} (u) \,d \mBM(u) \nonumber \\
        &= \int_{\partial^1 [y_0, \gamma. y_0)} 
        \frac{1}{ 2\tau \mub{\calG^F(v), r} } 
        \int_{\thickbs{\tau}{\calG^F(v), r} } 
        e^{\delta \beta(v^+, \pi(v), \pi(w))} \,d \mus_v(w) \,d\sigma_L(v).
    \end{align}
    Using the homeomorphism $\wss{\calG^F(v)} \times \RR \to \ws(v)$ that sends $(w', t)$ to $\calG_t (w')$, we have
    \begin{align*}
        d \mus_v(w) = e^{-\delta t} \,d\muss_{\calG^F(v)} (w')\, dt = 
        e^{-\delta \beta(v^+, \pi(w'), \pi(w) ) } \,d\muss_{\calG^F(v)} (w') \,dt.
    \end{align*}
    Hence we can compute
    \begin{align*}
        \int_{\thickbs{\tau}{\calG^F(v), r} } 
        e^{\delta \beta(v^+, \pi(v), \pi(w))} \,d \mus_v(w) = 
        \int_{\bss{\calG^F(v), r} \times [-\tau, \tau]} 
        e^{\delta \beta(v^+, \pi(v), \pi(w)) +
         \delta \beta(v^+, \pi(w), \pi(w') ) } \,d\muss_{\calG^F(v)} (w') \,dt.
    \end{align*}
    Since
    \begin{align*}
        \beta(v^+, \pi(v), \pi(w)) + \beta(v^+, \pi(w), \pi(w') ) =
        \beta(v^+, \pi(v), \pi(w')) = \beta(v^+, \pi(v), \pi(\calG^F(v))) = F(v),
    \end{align*}
    the above integral is equal to 
    $2\tau e^{\delta F(v)} \mub{\calG^F(v), r} $. Plugging this into Equation (\ref{interated-int}) concludes the proof of the proposition.
\end{proof}

\section{Adjusted counts} \label{count-control}

% Recall the notation introduced at the beginning of Section \ref{reduction-control}. 
Let $\gamma \in \Gamma$, and let $h \from \Gamma \to \RR$ be a $\gamma$--invariant function, that is, $h(\gamma g) = h(g)$ for all $g \in \Gamma$. For such a function, define
\begin{align*}
    \calB_\gamma^h (T) \coloneqq \{ \gmcycl. g \in \gmcycl \backslash \Gamma : h(g) \leq T \}
\end{align*}
to be the set of $\gmacycle$--orbits on which the value of $h$ is at most $T$. For the following theorem, recall that $-\clow^2$ is the lower bound on the curvature of $M$.
% For the following theorem, recall the definition of $\sigma(F_i)$, given in (\ref{sigmafdef}).
% \nextkpa{counting-kpa}
\begin{theorem} \label{control counting}
    Assume $M$ is two-dimensional, or $M$ is of dimension at least three and $\clow \leq \clowval$. Then for every $\alpha, \kappa > 0$, there exists a constant $\kappa'$, only depending on $\alpha, \kappa$, and the geometry of $M$, such that the following holds.
    Let $x \in \mtilde,\, \id \neq  \gamma \in \Gamma$, and $L$ denote the axis of $\gamma$. Let
    \begin{align*}
        F_1 \from \partial^1 L \to \RR, \quad \text{ and }\quad F_2 \from S(x) \to \RR
    \end{align*}
    be \hold{\alpha} functions, and assume that $F_1$ is $\gamma$--invariant. Let the $\gamma$--invariant function $h \from \Gamma \to \RR$ satisfy
    \begin{align} \label{h_eq}
        h(g) = d(gx, L) - F_1 (P^1_L(gx)) - F_2(g^{-1}. P^1_{gx} (L)) + \bfO ({e^{-\kappa d(gx, L) } }) 
    \end{align}
    for $g \in \Gamma$. Then for $T > 0$,
    \begin{align*}
        \# \calB^h_\gamma (T) = (1 + \bfO_{F_1, F_2} (e^{- \kappa' T} )  )\,
        \frac{e^{\delta T}}{\delta} \sigma_\gamma (F_1) \sigma_x(F_2),
    \end{align*}
    where $\sigma_\gamma (F_1)$ and $\sigma_x(F_2)$ are as in (\ref{sigmafdef_1}) and (\ref{sigmafdef_2})
\end{theorem}

Assuming the above theorem, we give a proof of Theorem \ref{main-theorem}.

\begin{proof} [Proof of Theorem \ref{main-theorem}] 
    We say that an element $g \in \Gamma$ is primitive if $g = \hat{g}^k$ (for some $\hat{g} \in \Gamma$) implies $k = \pm 1$. Let the primitive element $\gamhat$ be such that $\gamma = \gamhat^k$ for some $k \in \NN$. 
    % For an element $\gamma' \in \conj_\gamma$, there exists $g \in \Gamma$ such that $\gamma' = g^{-1} \gamma g$, and such a $g$ is unique upto left multiplication by a power of $\gamhat$. 
    Since
    \begin{align*}
        d(x, g^{-1}\gamma g .y) = d(gx, \gamma g.y),
    \end{align*}
    and the centralizer of $\gamma$ is $\gamhatcycl$, the map $g \mapsto g^{-1} \gamma g.y$ gives a one-to-one correspondance between the sets
    \begin{align*}
        \{\gamhatcycl \cdot g: d(gx, \gamma g.y) \leq T \}, \quad \text{ and } \quad
        B_T(x) \cap \conj_\gamma \cdot y.
    \end{align*} 
    Denoting the axis of $\gamma$ by $L$, by Proposition \ref{reduction-prop} there exists a $\gamma$--invariant function $F_1$ on $\partial^1 L$ and a function $F_2$ on $S(x)$, both \holder{}-continuous, such that for every $g \in \Gamma$ we have
    \begin{align*}
        d(gx, \gamma g. y) = 2 d(gx, L) + F_1 (P^1_L(gx)) + F_2(g^{-1}. P^1_{gx} (L)) + \bfO_{x, y, \gamma} (e^{- d(gx, L)/2 }).
    \end{align*}
    Note that by (\ref{F_in_lim}), $F_1$ is $\gamhat$--invariant, thus, applying Theorem \ref{control counting} to $\gamhat, -F_1, -F_2$ and $h(g) \coloneqq d(gx, \gamma g.y)/2$, the desired result follows for 
    $$\sigma \coloneqq \frac{1}{\delta}\, \sigma_{\gamhat} (\frac{-F_1}{2})\, \sigma_x (\frac{-F_2}{2}).$$
\end{proof}

%convention and fixed stuff:
\nextalfa{Fi-holder}
\nextkpa{h error}

For the rest of this section, we fix $\alfa{\ref{Fi-holder}} \coloneqq \alpha,\, \kpa{\ref{h error}} \coloneqq \kappa,\, x,\, \gamma,\, L,\, F_1,\, F_2$ to be as in Theorem \ref{control counting}, and denote $\sigma_\gamma (F_1)$ (resp. $\sigma_x(F_2)$) by $\sigma(F_1)$ (resp. $\sigma(F_2)$). We also fix the function $h$ to be as in this theorem, and denote $\calB^h_\gamma(T)$ by $\calH(T)$.
% From now on, we , and the value of $h$ at this orbit by $h([g])$. 
Using the action of deck transformations on $T^1 \mtilde$, we can define $F_2$ on $\cup_{g \in \Gamma} S(gx)$, so from now on we assume that $F_2$ is defined on this larger domain.
For $g \in \Gamma$, let $v_1(g)$ and $v_2(g)$ be tangents to $[gx, P_L (gx)]$ at $P_L(gx)$ and $gx$ respectively. 
With these definitions, the assumption (\ref{h_eq}) of Theorem \ref{control counting} can be written as 
\begin{align}
    h(g) = d(gx, L) - F_1 (v_1(g) ) - {F_2}(v_2(g)) + \bfO({e^{-\kpa{\ref{h error} } d(gx, L) }  }) \label{error of h},
\end{align}
and the statement of this theorem can be written as 
\begin{align*}
    \# \calH (T) = (1 + \bfO (e^{- \kappa' T} )  )\,
    \frac{e^{\delta T}}{\delta} \sigma (F_1) \sigma (F_2)  \quad \text{ for some } \kappa' > 0.
\end{align*}

Note that $h(g)$ differs from $d(gx, L)$ by a bounded amount which is determined, upto an exponentially small error, by the manner that the perpendicular from $gx$ to $L$ departs from $gx$ and enters $L$. We say that $h(g)$ is the distance between $gx$ and $L$, adjusted by \emph{adjustment functions} $F_1$ and $F_2$, or simply the \emph{adjusted height} of $gx$.

% We fix the adjustment functions $F_i$ for $i = 1, 2$, and the height function $h$ in the above theorem for the rest of this section.

Let $X$ be a topological space, $\varphi \from X \to \RR$ an arbitrary function, and $G \from X \to X$ a homeomorphism. Define
$ G \cdot \varphi \from X \to \RR$ by $G\cdot \varphi \coloneqq \varphi \circ G^{-1}$.
Note that if $H \from X \to X$ is another homeomorphism, then 
$(G \circ H) \cdot \varphi = G\cdot  (H \cdot \varphi)$. 
Since $\calG_t$ and $\iota$ are homeomorphisms from $T^1 \mtilde$ (resp. $T^1 M$) to itself, $\calG_t \cdot \varphi$ and $\iota \cdot \varphi$ can be defined for a function $\varphi \from T^1 \mtilde \to \RR$ (resp. $\varphi \from T^1 M \to \RR$). Associating to $g \in \Gamma$ its deck
transformation, $g \cdot \varphi$ can be defined for a function $\varphi \from T^1 \mtilde \to \RR$.
Let $\varphi$ and $\psi$ be arbitrary functions from $T^1 \mtilde$ to $\RR$, and assume that $\varphi$ is $\gamma$--invariant. Let $\bar{\varphi}_\gamma$ and $\bar{\psi}$ be defined by equations (\ref{pushdown-line}) and (\ref{pushdown-point}) respectively. Denoting $\gmacycle \cdot g$ by $[g]$, one can check that
\begin{align} \label{expandgen}
    \int_{T^1 M} \bar{\varphi}_\gamma \times \bar{\psi}\, d\mbarBM = \sum_{[g] \in \gmcycl \backslash \Gamma} \int_{T^1 \mtilde} \varphi \times g \cdot \psi\, d\mBM.
\end{align}

Fix a radius $R$ for the rest of this section, and using equation (\ref{defoftest}), define the test functions $\phi_\tau^i$ for $i = 1, 2$ by
\begin{align*}
    \phi^1_\tau \coloneqq \test{L}{F_1}{R, \tau},\,\,\, \text{ and }\,\,\, 
    \phi^2_\tau \coloneqq \test{x}{F_2}{R, \tau}.
\end{align*}
Plugging $\varphi \coloneqq \calG_t \cdot \phi^1_\tau$ and $\psi \coloneqq \iota \cdot \phi^2_\tau$ in (\ref{expandgen}) we obtain
% and using the definition of $j_\tau(g, t)$ (Equation (\ref{j-func-def})) we obtain
\begin{align} \label{expand-mixing}
    \int_{T^1 M} \calG_t \cdot \barphi^1_\tau \times \iota \cdot \barphi^2_\tau \,d\mbarBM = \sum_{[g] \in \gmcycl \backslash \Gamma} \int_{T^1 \mtilde} j_\tau(g, t)  \,d \mBM 
    = \sum_{[g] \in \gmcycl \backslash \Gamma} J_\tau(g, t),
\end{align}
where $j_\tau (g, t) \from T^1 \mtilde \to \RR$ and $J_\tau(g, t) \in \RR$ are defined by
\begin{align*}
    j_\tau (g, t) \coloneqq \calG_t \cdot \phi^1_\tau \times (\iota \cdot g \cdot \phi^2_\tau), \,\,\,\text{ and }\,\,\,
    J_\tau(g, t) \coloneqq  \int_{T^1 \mtilde}\, j_\tau(g, t) \,d\mBM.
\end{align*}
For $0 < T_1 < T_2$, multiplying left and right hand sides of (\ref{expand-mixing}) by $e^{\delta t}$ and integrating over $T_1 \leq t \leq T_2$, we obtain
\begin{align} \label{I in sum}
    I_\tau (T_1, T_2)  \coloneqq \int_{T_1}^{T_2} e^{\delta t} \int_{T^1 M} \calG_t \cdot \barphi^1_\tau \times \iota \cdot \barphi^2_\tau \,d\mbarBM \, dt = 
    \sum_{[g] \in \gmcycl \backslash \Gamma} 
    \int_{T_1}^{T_2} e^{\delta t} J_\tau(g, t) \,dt 
\end{align}
To prove Theorem \ref{control counting}, in the first step we show that the right hand side of the above, for $T_1 \coloneqq \frac{T}{2}$ and $T_2 \coloneqq T$, approximates the number of $\gamma$--orbits of adjusted height at most $T$. In the next step, we compute the left hand side of (\ref{I in sum}), using the mixing property of geodesic flow. For the first step we need Lemma \ref{j almost one}, and for the second step we need Lemma \ref{phi-mixing}. Assuming these lemmas, the details of the proof is given in Section \ref{detproof}. This method is similar to the one used in \cite{PPCommonPerp, roblin-thesis}.

\subsection{$J_\tau^\infty(g)$ is exponentially close to $1$}
Recall the notation introduced in Section \ref{test-intro}, and recall that we fixed a radius $R>0$ throughout the paper. The following is an adaptation of Lemma 7 of \cite{PPCommonPerp} to our setting. (See Item (\ref{t and t two}) at the beginning of Section \ref{test-intro}.)

\nextkpa{v-vprime}
\nextkpa{compare distance}

\begin{lemma} \label{pp-lemma7}
    There are constants $\kpa{\ref{v-vprime}}$ and $\kpa{\ref{compare distance}}$ such that the following holds.
    Let $\tau > 0$, and let $w_1 \in \thickbs{\tau}{\calG^{F_1} (\partial^1 L), R}$ be such that $w_2 \coloneqq \iota(\calG_t(w_1)) \in \thickbs{\tau}{\calG^{F_2} (S(gx)), R}$. Setting $v'_1 \coloneqq p_L(w_1)$ and $v'_2 \coloneqq p_{gx} (w_2)$, we have
    \begin{align} 
        d( v_i(g), v'_i) &= \bfO (e^{- \kpa{\ref{v-vprime}} d(gx, L)}),\,\,\, \text{ for } i = 1, 2; \label{vvprime-close}\\
        d(gx, L) &= t + {F_1}(v_1(g)) + {F_2}(v_2(g)) + \bfO({\tau + e^{- \kpa{\ref{compare distance}} d(gx, L)}}).  \label{d-in-t}
    \end{align}
\end{lemma}

\begin{proof}
    Let $t' \coloneqq t + F_1 (v'_1) + F_2 (v'_2)$ and $w' \coloneqq \calG_{\frac{t'}{2} - F(v'_1)} w_1$.
    Then for some $r = \bfO(R)$ we have
    \begin{align*}
        \iota(\calG_{\frac{t'}{2}} w') \in \thickbs{\tau}{S(x), r},\,\,\,\, \text{ and }\,\,\,\,  \calG_{- \frac{t'}{2}} w' \in \thickbs{\tau}{\partial^1 L, r}.
    \end{align*}
    Let $h \coloneqq P_L(gx)$ be the foot of perpendicular from $gx$ to $L$. By the third item in \cite{PPCommonPerp} Lemma 7, there exists $y \in [gx, h]$ with $d(\pi(w'), y) = \bfO(e^{-\frac{t}{2}})$. By the same item (of the same lemma) 
    % for the point $y' \coloneqq \calG_{\frac{t'}{2}} v'_1$ in the geodesic ray determined by $v'_1$ 
    we have $d(\pi(w'), y') = \bfO(\tau + e^{-\frac{t}{2}})$ for $y' \coloneqq \calG_{\frac{t'}{2}} v'_1$. Triangle inequality then implies $d(y, y') = \bfO(\tau + e^{-\frac{t}{2}}) = \bfO(1)$. 
    Since $d(y', \pi(v'_1)) = \frac{1}{2} d(gx, L) + \bfO(1)$, Proposition \ref{proj-close} implies 
    $d( v_1(g), v'_1) = \bfO (e^{- \kpa{\ref{v-vprime}} d(gx,L)})$ for $\kpa{\ref{v-vprime}} \coloneqq \frac{1}{2}$. This proves (\ref{vvprime-close}) for $i = 1$. The proof for $i = 2$ is analogous.

    % $d(\pi(w'), \calG_{\frac{t'}{2}} v'_1) = \bfO(\tau + e^{-\frac{t}{2}})$
    The first item of \cite{PPCommonPerp} Lemma 7 implies that 
    \begin{align*}
        d(gx, L) = t' + \bfO(\tau + e^{-\frac{t}{2}}) = 
        t + F_1 (v'_1) + F_2 (v'_2) + \bfO(\tau + e^{-\frac{t}{2}})
    \end{align*} 
    Since $d( v_i(g), v'_i) = \bfO (e^{-\kpa{\ref{v-vprime}} d(gx, L)})$ and $F_i$ is \alfa{\ref{Fi-holder}}--\holder{} we have 
    $F_i(v'_i) = F_i(v_i(g)) + \bfO(e^{-\kappa d(gx, L)})$ for 
    $\kappa \coloneqq \kpa{\ref{v-vprime}} \alfa{\ref{Fi-holder}}$. Since $t = d(gx, L) + \bfO(1)$, we have $\bfO(\tau + e^{-\frac{t}{2}}) = \bfO(\tau + e^{-\frac{d(gx, L)}{2}})$. Equation (\ref{d-in-t}) now follows for $\kpa{\ref{compare distance}} \coloneqq \min \{\frac{1}{2}, \kappa\}$.
\end{proof}

\nextkpa{h-equals-t}
\begin{corollary} \label{h-almost-t}
    There exists a constant $\kpa{\ref{h-equals-t}}$ such that the following holds.
    Let $\tau > 0$, $g \in \Gamma$, and $t > 0$ be such that $j_\tau(g, t) \neq 0$. Then we have
    \begin{align*}
        h(g) = t + \bfO(\tau + e^{- \kpa{\ref{h-equals-t}} d(gx, L)}).
    \end{align*}
\end{corollary}

\begin{proof}
    Let $u$ be such that $j_\tau(g, t) (u) \neq 0$, and 
    set $w_1 \coloneqq \calG_{-t} u$ and $w_2 \coloneqq \iota(u)$.
    Then $w_1$ and $w_2$ satisfy the assumptions of Lemma \ref{pp-lemma7}, hence Equation (\ref{d-in-t}) implies
    % by Lemma \ref{distance to L} we have
    \begin{align*}
        d(gx, L) - {F_1}(v_1(g) ) - {F_2}(v_2(g) ) = t   + \bfO(\tau + e^{- \kpa{\ref{compare distance}} d(gx, L)}).
    \end{align*}
    This, combined with (\ref{error of h}), implies the lemma for  $\kpa{\ref{h-equals-t}} \coloneqq \min\{ \kpa{\ref{h error}}, \kpa{\ref{compare distance}} \}$.
\end{proof}

Define
\begin{align*}
    % J_\tau(g, t) =  \int_{T^1 \mtilde} j_\tau(g, t)(u) d\mBM(v)\\
    J^\infty_\tau(g) \coloneqq \int_0^{\infty} e^{\delta t} J_\tau(g, t) \,dt.
\end{align*}
The following lemma is proved in Section 3.2. of \cite{PPCommonPerp} (see Equation (23) of that paper), when $F_1$ and $F_2$ are both zero functions and with a slightly different definition for $\bss{u, r}$. (See Remark \ref{two-ham}.) The proof is completely analogous in our case, hence we only provide a sketch of the proof.

\nextkpa{J-almost-one-kpa}
% \nextalfa{J-almost-one-alpha}
\begin{lemma} \label{j almost one}
    Assume $M$ is two-dimensional, or $M$ is of dimension at least three and $\clow \leq \clowval$.
    Then there exists a constant $\kpa{\ref{J-almost-one-kpa}}$ such that for all $g \in \Gamma$ and $\tau \geq e^{- \kpa{\ref{J-almost-one-kpa}} d(gx, L) }$ we have 
    \begin{align*}
        J_\tau^\infty (g) = 1 + \bfO(\tau).
    \end{align*}
\end{lemma}

\begin{proof} 
    We fix $g$ throughout the proof. We also fix $\tau \geq e^{-\hat{\kappa} d(gx, L)}$ for some small enough $\hat{\kappa}$ to be optimized in the course of the proof. We say that two vectors $u, v \in T^1 \mtilde$ are exponentially close if $d(u, v) = \bfO(e^{- \kappa d(gx, L)})$ for some constant $\kappa$. We make a similar definition for real numbers.

    If $J_\tau(g, t) \neq 0$ for some $t$, then by Corollary \ref{h-almost-t} (and imposing $\kpahat \leq \kpa{\ref{h-equals-t}})$) we have $t = h(g) + \bfO(\tau)$, therefore
    \begin{align*}
        J^\infty_\tau(g) = (1 + \bfO(\tau))\, e^{\delta h(g)} \int_0^{\infty} J_\tau(g, t) \,dt = \int_{T^1 \mtilde \times \RR} j(g, t) (u)\, d\mBM\, dt.
    \end{align*}
    For $D = L$ or $D$ a point in $\mtilde$ define $p_D^F(u) \coloneqq \calG^F(p_D(u))$.
    Define $v^F_i \coloneqq \calG^{F_i}(v_i(g))$ for $i = 1, 2$.
    If $j_\tau(g, t) (u) \neq 0$ for some $u$, then by Lemma \ref{pp-lemma7} $v_1(g)$ and $p_L u$ are exponentially close, hence by Lemma \ref{gt-holder}, $v^F_1 \coloneqq \calG^F(v_1(g))$ and $p_L^F(u)$ are exponentially close. by Theorem \ref{radius-hold-cont}, $M$ has the RHC property, thus Lemma \ref{mub-holder} implies $\mub{v_1^F, R}$ and $\mub{p_L^F(u), R}$ are exponentially close. By a similar argument, $\mub{v_2^F, R}$ and $\mub{p_{gx}^F(\iota(u)), R}$ are also exponentially close. As a result, there exists $\kappa$ such that for all $u \in \sprt(j_\tau(g, t))$ we have
    \begin{align*}
        j_\tau(g, t)(u) = \frac{1}{4 \tau^2 \mub{p_L^F(u), R} \mub{p_{gx}^F(\iota(u)), R}} = (1 + \bfO(e^{- \kappa d(gx, L)}))\, \frac{1}{4\tau^2 \mub{v_1^F, R} \mub{v_2^F, R}}.
    \end{align*}
    Let $S_g$ denote the support of $j_\tau(g, \param)(\param)$, that is,
    \begin{align*}
        S_g \coloneqq \{(u, t): j_\tau(g, t)(u) \neq 0 \}.
    \end{align*}
    We have
    \begin{align} \label{J-in-char}
        J^\infty_\tau(g) = (1 + \bfO(e^{- \kappa d(gx, L)} + \tau))\, e^{\delta h(g)}\, \frac{1}{4\tau^2 \mub{v_1^F, R} \mub{v_2^F, R}} \int_{T^1 \mtilde \times \RR} \mathbbm{1}_{S_g}(u, t) \, d\mBM(u).\, dt.
    \end{align}
    Imposing $\kpahat \leq \kappa$, we can replace the term $\bfO(e^{- \kappa d(gx, L)} + \tau)$ in the above equation by $\bfO(\tau)$.

    For $u \in T^1 \mtilde$ and for $i = 1, 2$ define $u'_i \in \wss{v_i^F}$ such that $(u'_1)^- = u^-$ and $(u'_2)^- = u^+$, and let $s \coloneqq d(P_{(u^-, u^+)} o, \pi(u))$ be the time parameter in the Hopf parametrization $u$ (see Section \ref{test-intro}). The map $u \mapsto (u'_1, u'_2, s)$ is a homeomorphism from $T^1 \mtilde \setminus \{u: u^- = (v_1^F)^+ \,\text{ or }\, u^+ = (v_2^F)^+ \}$ to $\wss{v_1^F} \times
    \wss{v_2^F} \times \RR$. Note that if $d(gx, L)$ is large enough and $j_\tau(g, t)(u) \neq 0$, then $u \in T^1 \mtilde \setminus \{u: u^- = (v_1^F)^+, \,\text{ or }\, u^+ = (v_2^F)^+ \}$. We identify $u$ with $(u'_1, u'_2, s)$ for the rest of this proof.
    
    Define $\mathbf{B}_g \subset \wss{v_1^F} \times \wss{v_2^F}$ by
    \begin{align*}
        \mathbf{B}_g \coloneqq \{(u'_1, u'_2): j_\tau(g, t)(u'_1, u'_2, s) \neq 0 \text{ for some } t, s \}.
    \end{align*}
    By the properties of $H^{\param \param}$ introduced in Lemma \ref{mub-holder}, there exists $\kappa'$ such that for $\epsilon \coloneqq e^{-\kappa' d(gx, L)}$ we have
    \begin{align} \label{bg-include}
        \bss{v^F_1, r - \epsilon} \times \bss{v^F_2, r - \epsilon} \subset
        \mathbf{B}_g \subset \bss{v^F_1, r + \epsilon} \times \bss{v^F_2, r + \epsilon}.
    \end{align}
    For $(u'_1, u'_2) \in \mathbf{B}_g$, set
    \begin{align*}
        P_{(u'_1, u'_2)} \coloneqq \{(t, s): j_\tau(g, t)(u'_1, u'_2, s) \neq 0  \}.
    \end{align*}
    Let $P_\tau$ denote the rhombus $\{(t, s): |s| \leq \tau \text{ and } |t - s| \leq \tau \}$. It can be proved that $P_{(u'_1, u'_2)}$ is always a translation of $P_\tau$. More precisely, there are functions $\bar{t}$ and $\bar{s}$ from $\mathbf{B}_g$ to $\RR$ such that
    \begin{align} \label{rhombus-trans}
        P_{(u'_1, u'_2)} = \big(\bar{t}(u'_1, u'_2) , \bar{s}(u'_1, u'_2)\big) + P_\tau \,\text{ for }\, (u'_1, u'_2) \in \mathbf{B}_g.
    \end{align}
    Arguing as in Lemma 10 of \cite{PPCommonPerp}, for $\hat{\kappa}$ small enough we have
    \begin{align*}
        d\mBM(u) = (1 + \bfO(\tau))\, e^{-\delta h(g)} d\muss_{v_1^F}(u'_1)\, d\muss_{v_2^F}(u'_2)\, ds ,
    \end{align*}
    therefore
    \begin{align*}
        \int_{T^1 \mtilde \times \RR} \mathbbm{1}_{S_g}(u, t) \, d\mBM(u)\, dt &=
        (1 + \bfO(\tau))\, e^{-\delta h(g)} \int_{\wss{v_1^F} \times
        \wss{v_2^F} \times \RR \times \RR} \mathbbm{1}_{S_g} (u'_1, u'_2, s, t) \times \\
        & \hspace{7 cm} d\muss_{v_1^F}(u'_1)\, d\muss_{v_2^F}(u'_2)\, ds\, dt \\
        &= 4 \tau^2 e^{-\delta h(g)} \int_{\wss{v_1^F} \times \wss{v_2^F}} \mathbbm{1}_{\mathbf{B}_g}\, (u'_1, u'_2)\, d\muss_{v_1^F}(u'_1)\,\, d\muss_{v_2^F}(u'_2) \tag{by (\ref{rhombus-trans})} \\
        &= 4 \tau^2 e^{-\delta h(g)} (1 + \bfO(\tau))\, \mub{v_1^F, R} \mub{v_2^F, R}.
    \end{align*}
    For the last equality, we used (\ref{bg-include}) and Theorem \ref{radius-hold-cont}, and assumed $\kpahat$ is small enough.
    Plugging the above into (\ref{J-in-char}) concludes the proof for $\kpa{\ref{J-almost-one-kpa}} \coloneqq \kpahat$.
\end{proof}

\subsection{The exponential decay of correlations for test functions}

For a real valued \hold{\alpha} function $\varphi$ on a metric space $(X, d)$, define
\begin{align*}
    \norm{\varphi}_\alpha \coloneqq \max 
    \big\{ \norm{\varphi}_\infty,\, \sup_{x \neq y \in X} \frac{|\varphi(x) - \varphi(y)|}{d(x, y)^\alpha}
    \big\}.
\end{align*}
We say that the geodesic flow $\calG_t$ (on $M$) is \emph{exponentially mixing for \holder{} regularity $\alpha$} (for Bowen-Margulis measure), if there are constants $\kappa = \kappa(\alpha)$ and $c = c(\alpha)$ such that for all $\alpha$--\holder{} functions  $\varphi_1, \varphi_2 \from T^1 M \to \RR$ and $t > 0$ we have
\begin{align} \label{exp. mixing}
    \int_{T^1 M} \calG_t \cdot \varphi_1 \times \varphi_2\, dm_\text{BM} = 
    \int_{T^1 M} \varphi_1\, dm_\text{BM} \int_{T^1 M} \varphi_2\, dm_\text{BM}
    + \bfO( c e^{- \kappa t })\, \norm{\varphi_1}_\alpha\, \norm{\varphi_2}_\alpha.
\end{align}
If $\calG_t$ is exponentially mixing for some \holder{} regularity $\alpha_0$, (or for $C^{n_0}$ functions equipped with Sobolov norm of of order $n_0$), then it is exponentially mixing for all \holder{} regularity $\alpha$ (and for $C^n$ functions equipped with Sobolev norm of order $n$, for all positive integers $n$). See Lemma B.1. of \cite{Bernouli-mixing} for a proof. If $M$ is two-dimensional, or if $M$ is of dimension at least three and $\frac{1}{9}$--pinched (that is, $\clow \leq 3$), the geodesic flow is exponentially mixing. (See \cite{dolg-exp-mixing} and \cite{glp-exp-mixing}.)
 
The functions $\barphi^i_\tau$, introduced at the beginning of this section, are not \holder{}-continuous; in fact, they are not even continuous. However, an equation similar to (\ref{exp. mixing}) holds for these functions. (See Lemma \ref{phi-mixing}.) The reason is that $\barphi^i_\tau$ can be well-approximated by \holder{}-continuous functions, with some control over the \holder{} constants.
% The trick is to use \holder{}-continuous functions that well-approximate $\barphi^i_\tau$, with some control over the \holder{} constants. 
Such functions are constructed in Appendix \ref{construct-chi}, and a similar construction is sketched in Section 6 of \cite{PPSkin}. We use the following elementary lemma in the proof of Lemma \ref{phi-mixing}.

\begin{lemma} \label{elementary}
    Let $(X, \mu)$ be a measured space and $\varphi,\, \psi,\, \varphi',\, \psi'$ be measurable functions with finite $L^1$ and $L^\infty$ norms. Then we have
    \begin{align*}
        |\langle \varphi, \psi \rangle - \langle \varphi', \psi' \rangle| \leq
        \norm{\varphi - \varphi'}_{L^1} \norm{\psi}_{L^\infty} +
        \norm{\varphi'}_{L^\infty} \norm{\psi' - \psi}_{L^1},
    \end{align*}
    where the inner product $\langle \varphi, \psi \rangle$ is defined to be $\int_X \varphi \psi \, d\mu$.
\end{lemma}

% \nextconst{phi mixing const}
% \nextkpa{phi mixing exp}
\nextkpa{phi-mixing-kpa}
\begin{lemma} \label{phi-mixing}
    Assume $M$ is two-dimensional, or $M$ is of dimension at least three and $\clow \leq 3$. Then there exists a constant $\kpa{\ref{phi-mixing-kpa}}$ such that for all $t > 0$ and $\tau \geq e^{-\kpa{\ref{phi-mixing-kpa}} t}$ we have
    \begin{align*}
        \int_{T^1 M} \calG_t \cdot \barphi^1_\tau \times \iota \cdot \barphi^2_\tau\, d\mbarBM = 
        (1 + \bfO(
            \tau
        ))\, \sigma_\gamma (F_1) \sigma_x (F_2).
    \end{align*}
\end{lemma}

\begin{proof}
    To simplify the notation, we will drop $d \mbarBM$ when we integrate over $T^1 M$.
    Let $\alfa{\ref{chi-holder}}$, $\alfa{\ref{chiplus-alpha}}$, and $n_1$ be the constants given by Lemma \ref{smoothing}. Then, by this lemma, for every $\epsilon < \tau^{n_1}$ there are functions $\barchi_i(\tau, \epsilon) \coloneqq \barchi_{F_i}(R, \tau, \epsilon)$, for $i = 1, 2$, such that 
    \begin{enumerate} [(i)]
        \item $0 \leq \barchi_i(\tau, \epsilon) \leq \barphi^i_\tau$.
        \item $\barchi_i(\tau, \epsilon)$ is \hold{\alfa{\ref{chi-holder}}} with $\alfa{\ref{chi-holder}}$--norm of order 
        $\frac{1}{\tau \epsilon}$.
        %  $\norm{\chi(r, \tau, \epsilon)}_{\alfa{\ref{chi-holder}}} = 
        % \bfO(\frac{1}{\tau \epsilon})$
        \item \label{lone-in-text} $\norm{\barphi^i_\tau - \barchi_i(\tau, \epsilon)}_{L^1} =
        \bfO(\frac{
            \epsilon^{\alfa{\ref{chiplus-alpha}}}
        }{\tau} )$.
        % , where $\norm{\param}_{L^1}$ denotes the $L^1$--norm.
    \end{enumerate}

    Since $\iota$ is an isometry with respect to Sassaki metric, $\iota \cdot \barchi_2(\tau, \epsilon)$ is \hold{(\alfa{\ref{chi-holder}}, \bfO(\frac{1}{\tau \epsilon}))}. 
        Writing (\ref{exp. mixing}) for $\alpha \coloneqq \alfa{\ref{chi-holder}}$, we obtain $\kappa' = \kappa'(\alfa{\ref{chi-holder}})$ and $ c = c (\alfa{\ref{chi-holder}})$ such that for all $t > 0$ we have 
        \begin{align*}
            \int_{T^1 M} \calG_t \cdot \barchi_1(\tau, \epsilon ) \times \iota \cdot \barchi_2(\tau, \epsilon) =
        \int_{T^1 M} \barchi_1(\tau, \epsilon ) 
        &\times
        \int_{T^1 M} \iota \cdot \barchi_2(\tau, \epsilon)\\
        &+ \bfO(c e^{-\kappa' t})\,
        \norm{\barchi_1(\tau, \epsilon )}_{\alfa{\ref{chi-holder}}}\,
        \norm{\iota \cdot \barchi_2(\tau, \epsilon)}_{\alfa{\ref{chi-holder}}}.
    \end{align*}
    Since $\iota \from T^1 M \to T^1 M$ is measure preserving, $\int_{T^1 M} \iota \cdot \barchi_2(\tau, \epsilon) = \int_{T^1 M} \barchi_2(\tau, \epsilon)$, thus
    \begin{align} \label{mix-chi}
        \int_{T^1 M} \calG_t \cdot \barchi_1(\tau, \epsilon ) \times \iota \cdot \barchi_2(\tau, \epsilon)
        = \int_{T^1 M} \barchi_1(\tau, \epsilon) \times \int_{T^1 M} \barchi_2(\tau, \epsilon) + \bfO(
            \frac{ e^{-\kappa' t}}{\tau^2 \epsilon^2}
        ).
    \end{align}
    Item (\ref{lone-in-text}) above implies that
    \begin{align} \label{chi-phi}
        \int_{T^1 M} \barchi_i (\tau, \epsilon ) = 
        \int_{T^1 M} \barphi_i (\tau) + 
        \bfO(\frac{
                \epsilon^{\alfa{\ref{chiplus-alpha}}}
            }{\tau} )  =
        \big(1 + 
            % \bfO(\tau))
            \bfO(\frac{
                \epsilon^{\alfa{\ref{chiplus-alpha}}}
            }{\tau} )
        \big) \, \sigma(F_i).
    \end{align}
    After using Lemma \ref{elementary}, Item (\ref{lone-in-text}) above, coupled with the fact that both $\norm{\barphi^i_\tau}_{L^\infty}$ and $\norm{\barchi_i(\tau, \epsilon)}_{L^\infty}$ are of order $\frac{1}{\tau}$, implies
    \begin{align*}
        %by elementary lemma:
        \int_{T^1 M} \calG_t \cdot \barchi_1(\tau, \epsilon) \times
        \iota \cdot \barchi_2 (\tau, \epsilon) =  
        \int_{T^1 M} \calG_t \cdot \barphi^1_\tau \times
        \iota \cdot \barphi^2_\tau  +
        \bfO(\frac{
                \epsilon^{\alfa{\ref{chiplus-alpha}}}
            }{\tau^2}). 
    \end{align*}
    Plugging the above equation and (\ref{chi-phi}) into (\ref{mix-chi}), we get 
    \begin{align*}
        \int_{T^1 M} \calG_t \cdot \barphi^1_\tau \times
        \iota \cdot \barphi^2_\tau = \big(1 + 
        \bfO(
            \frac{ e^{-\kappa' t}}{\tau^2 \epsilon^2}
            + \frac{
                \epsilon^{\alfa{\ref{chiplus-alpha}}}
            }{\tau} +
            \frac{
                \epsilon^{\alfa{\ref{chiplus-alpha}}}
            }{\tau^2}
        )
        \big)\, \sigma(F_1) \sigma(F_2).
    \end{align*}
    Let $n \coloneqq 
        \max \{ n_1, \frac{3}{
            {\alfa{\ref{chiplus-alpha}}}
            } \}$.
    Setting $\epsilon \coloneqq \tau^n$ in the above equation gives
    \begin{align*}
        \int_{T^1 M} \calG_t \cdot \barphi^1_\tau \times
        \iota \cdot \barphi^2_\tau = 
        \big(1 + \bfO (
            \frac{ e^{-\kappa' t}}{\tau^{2n + 2}} +
            \tau )\big) \, \sigma(F_1) \sigma(F_2).
    \end{align*}
    Let $\kpa{\ref{phi-mixing-kpa}}$ be a small enough constant such that 
    $\kpa{\ref{phi-mixing-kpa}}(2n + 3) < \kappa' $. Then we have
    \begin{align*}
        \frac{e^{-\kappa' t}}{\tau^{2n + 2}} < \tau,
    \end{align*}
    hence we can replace the term 
    \begin{align*}
        \bfO(
        \frac{ e^{-\kappa' t}}{\tau^{2n + 2}} +
        \tau)
    \end{align*}
   in the last Equation by $\bfO({ }\tau)$. 
    This concludes the proof of the lemma.
\end{proof}

\subsection{Proof of Theorem \ref{control counting}} \label{detproof}

For $0 < T_1 < T_2$, recall the definition of $I_\tau(T_1, T_2)$ given in (\ref{I in sum}).

\nextkpa{I-et-kpa}
\begin{lemma} \label{I closed formula}
    Assume $M$ is two-dimensional, or $M$ is of dimension at least three and $\clow \leq 3$. Then there exists a constant $\kpa{\ref{I-et-kpa}}$ such that the following holds. Let $a > 0$, then there exists a constant $c = c(a)$ such that for all $T \geq 0$, $\tau \geq e^{-\kpa{\ref{I-et-kpa}}T}$, and real number $a'$ with $|a'| \leq a$ we have
    \begin{align*}
        I_\tau (\frac{T}{2}+a', T) = (1 + \bfO(c \tau))\, \cte e^{\delta T}.
    \end{align*}
\end{lemma}

\begin{proof}
    Let $\kpa{\ref{phi-mixing-kpa}}$ be the constant given by Lemma \ref{phi-mixing}. Choose $\kpa{\ref{I-et-kpa}} < 
    \min \{\frac{\kpa{\ref{phi-mixing-kpa}}}{2}, \frac{\delta}{2}\}$ and, without loss of generality, assume $T$ is large enough.
    Since $\kpa{\ref{I-et-kpa}} < \frac{\kpa{\ref{phi-mixing-kpa}}}{2}$,
    for $t \geq \frac{T}{2} + a'$ we have $\tau \geq e^{-\kpa{\ref{phi-mixing-kpa}}t}$, hence Lemma \ref{phi-mixing} implies that for such $t$,
    %  for  $\frac{T}{2} + a' \leq t \leq T$ we have
    \begin{align*}
        &\int_{T^1 M} \calG_t \cdot \barphi^1_\tau \times \iota \cdot \barphi^2_\tau\, d\mbarBM = (1 + \bfO( \tau))\, \sigma(F_1) \sigma(F_2) \\ &\implies
        I_\tau(\frac{T}{2} + a', T) = 
        % \int_{\frac{T}{2} + a'}^T e^{\delta t} \Sigma_\tau (t) dt= 
        (1 + \bfO( \tau))\, \sigma(F_1) \sigma(F_2)\, \int_{\frac{T}{2} + a'}^T e^{\delta t}\, dt.
    \end{align*}
    Letting $c' \coloneqq e^{\delta a}$, we get
    \begin{align*}
        \int_{\frac{T}{2} + a'}^T e^{\delta t} dt = \frac{e^{\delta T}}{\delta} \big(1 + \bfO( c'
            e^{-\frac{\delta}{2}T}
        )\big),
    \end{align*} 
    thus
    \begin{align*}
        I_\tau(\frac{T}{2} + a', T) = 
        \big(1 + \bfO( \tau + c'e^{-\frac{\delta}{2}T}) \big) 
        \, \frac{\sigma(F_1) \sigma(F_2)}{\delta} e^{\delta T}.
    \end{align*}
    Since $\kpa{\ref{I-et-kpa}} < \frac{\delta}{2}$, the lemma follows.
\end{proof}

Recall the definitions of the adjusted height function $h(g)$ and counting sets $\calH(T)$ given at the beginning of this section. The following upper bound on $\# \calH(T)$ is the last ingredient of the proof of Theorem \ref{control counting}.
\begin{lemma} \label{upper calh}
    For $T \geq 0$, we have 
    $\# \calH(T) = \bfO(e^{\delta T})$.
\end{lemma}

\begin{proof}
    Recall that the origin $o$ belongs to the axis of $\gamma$, denoted by $L$. Define
    \begin{align*}
        \calH_\Delta(T) \coloneqq \{g x: [g] \in \calH(T),\, \text{ and }\, P_L(gx)\in [o, \gamma o)\},
    \end{align*}
    and note that we have $\# \calH(T) = \# \calH_\Delta (T)$.
    Given $gx \in \calH_\Delta(T)$ with large enough distance from $L$, by Equation (\ref{error of h}) we have $d(gx, L) \leq T + M_1 + M_2 + 1$, where $M_i$ is the supremum of $|F_i|$ for $i = 1, 2$. 
    By triangle inequality we have,
    $$d(o, gx) \leq d(o, p_L(gx)) + d(gx, L) \leq T + c$$
    for $c \coloneqq d(o, \gamma o) + M_1 + M_2 + 1$. Denoting the ball of radius $R$ centered at $o$ in $\mtilde$ by $\calB_R(o)$, the above inequality implies that $\calH_\Delta(T) \subset \Gamma \cdot x \cap \calB_{T + c}(o)$.
    Since $\# (\Gamma \cdot x \cap \calB_t(o)) = \bfO(e^{\delta t})$, the result follows.
    % \sim \mu_o(\bdry) e^{\delta t}
\end{proof}

\begin{proof} [Proof of Theorem \ref{control counting}]
    Choose $\kappa'$ to be strictly less than $\min\{\frac{\kpa{\ref{h-equals-t}}}{2}, \frac{\kpa{\ref{J-almost-one-kpa}}}{2}, \frac{\delta}{2}  \}$, fix $T$ to be large enough, and let $\tau \coloneqq e^{-\kappa' T}$. For $0 <T_1 < T_2$, set
    \begin{align*}
        \calH(T_1, T_2) \coloneqq \{[g]: T_1 \leq h([g]) \leq T_2 \}.
    \end{align*}
    Since $\kappa' < \frac{\kpa{\ref{J-almost-one-kpa}}}{2}$, by Lemma \ref{j almost one} there exists a constant $c$ such that for every $[g] \in \calH(\frac{T}{2}, T)$ we have
    \begin{align*}
        1 - c\tau \leq J_\tau^\infty(g) \leq 1 + c\tau.
    \end{align*}

    For an element $g \in \Gamma$, let $t_{\max} (g)$ and $t_{\min}(g)$ denote the supremum and infimum of the set $\{t \geq 0: j_\tau (g, t) \neq 0  \}$ respectively. Since $\kappa' \leq \frac{\kpa{\ref{h-equals-t}}}{2}$, by Corollary \ref{h-almost-t} there exists a constant $c'$ such that for all $[g] \in \calH(\frac{T}{2}, T)$ we have
    \begin{align*}
        t_{\max}(g) \leq T + c' \tau,\, \text{ and }\, 
        t_{\min}(g) \geq \frac{T}{2} - c' \tau.
    \end{align*}
    % $t_{\max}(g) \leq T + c \tau$, 
    Hence, for such $[g]$, the term 
    \begin{align*}
        J_\tau^\infty(g) = 
        \int_0^{\infty} e^{\delta t} J_\tau(g, t)\, dt = 
        \int_{t_{\min}(g)}^{t_{\max}(g)} 
        e^{\delta t} J_\tau(g, t)\, dt = 
        \int_{\frac{T}{2} - c'\tau}^{T + c'\tau} 
        e^{\delta t} J_\tau(g, t)\, dt
    \end{align*}
    contributes to the sum (\ref{I in sum}) written for $I_\tau(\frac{T}{2} - c'\tau, T + c'\tau)$. As a result,
    \begin{align*}
        I_\tau(\frac{T}{2} - c'\tau, T + c'\tau) \geq 
        \sum_{[g] \in \calH(\frac{T}{2}, T)} J_\tau^\infty(g) \geq
        (1 - c\tau)\, \# \calH(\frac{T}{2}, T).
    \end{align*}
    % where the second inequality uses ??. 
    By Lemma \ref{I closed formula}, we obtain $c''$ such that 
    \begin{align*}
        \# \calH(\frac{T}{2}, T) \leq (1 + c''\tau)\, \cte e^{\delta T}.
    \end{align*}
    In a similar way we can find constants $c^{(3)}$ and $c^{(4)}$ such that
    \begin{align*}
        &I_\tau(\frac{T}{2} + c^{(3)} \tau, T - c^{(3)} \tau) \leq
         \sum_{[g] \in \calH(\frac{T}{2}, T)} J^\infty_\tau(g) \leq 
         (1 + c\tau)\, \# \calH(\frac{T}{2}, T) \\ &\implies
         (1 - c^{(4)} \tau)\, \cte e^{\delta T} \leq \# \calH(\frac{T}{2}, T).
    \end{align*}
    By Lemma \ref{upper calh}, 
    \begin{align*}
        \# \calH(T) = \# \calH(\frac{T}{2}) + \# \calH(\frac{T}{2}, T) = \# \calH(\frac{T}{2}, T) + \bfO(e^{\frac{\delta}{2} T}),
    \end{align*}
    hence the proposition follows from the above upper and lower bounds on $\# \calH(\frac{T}{2}, T)$. 
\end{proof}

% %%%%%%%%%%%%%%%%%%%%%%%%%%%%%%%%%%%%
% %%%%%%%%%%%%%%%%%%%%%%%%%%%%%%%%%%%%

% \appendix
% \section{A smoothing argument}

\section{Radius \holder{}-continuity of Bowen-Margulis measure} \label{rhc-section}

{\bf Definitions of RC and RHC.}
Radius \holder{}-continuity of Bowen-Margulis measures of strongly stable (unstable) balls, referred to as RHC throughout the text (see below for the precise definition), was first introduced in \cite{PPCommonPerp} as a technical assumption to obtain exponentially small error terms. However, in that paper, this condition was only verified for compact manifolds of constant negative curvature, in which the measure of a strongly stable ball is a smooth function of its radius. The goal of this section is to prove the RHC condition under more relaxed assumptions. (See Theorem \ref{radius-hold-cont}.) We adopt the notation introduced in Section \ref{test-intro}, and for $v \in T^1 \mtilde$, define the distance $\dss_v$ on $\wss{v}$ by 
$$\dss_v(u, u') \coloneqq d_{\mtilde}(\pi(u), \pi(u')).$$  
Recall the definition of strongly stable ball of radius $r > 0$ centered at $v$,
$$ \bss{v, r} = \{u \in \wss{v}: \dss_v(v, u) \leq r\},$$
and for a set $S \subset \wss{v}$, define
$$\bss{S, r} \coloneqq \bigcup_{u \in S} \bss{u, r}.$$
% denoted by $\bss{v, r}$ (see Section \ref{test-intro}), 
Define the strongly stable sphere of radius $r$ centered at $v$ by 
\begin{align*}
    \sss(v, r) \coloneqq \{u \in \wss{v}: \dss_v(v, u) = r \}.
\end{align*}
\begin{proposition} \label{rc_cond}
    The following are equivalent.
    \begin{enumerate} [(i)]
        \item \label{rc1} $\muss_v(\sss(v, r)) = 0$ for all $v \in T^1 \mtilde$ and $r > 0$.
        \item \label{rc2} For every $v \in T^1 \mtilde$, $\muss_v(\bss{v, r})$ is a continuous function of $r>0$.
        \item \label{rc3} $(v, r) \mapsto \muss(\bss{v, r})$ is a continuous function on $T^1 \mtilde \times \RR^{> 0}$.
    \end{enumerate}
\end{proposition}
\begin{proof}
     $(\ref{rc1}) \iff (\ref{rc2})$ and $(\ref{rc3}) \implies (\ref{rc2})$ are obvious. $(\ref{rc2}) \implies (\ref{rc3})$ follows from Lemma 1.16 of \cite{roblin-thesis}.
\end{proof}

% To obtain ?? from ?? see \cite{roblin-thesis} ??.
\begin{definition}
    If the items in Proposition \ref{rc_cond} are satisfied, we say that the strongly stable Bowen-Margulis measure of balls is radius-continuous, or in short, the RC condition is satisfied. 
\end{definition}

For constants $\alpha, c > 0$ and $v\in T^1 \mtilde$ and $r > 0$, we say that $\sss(v, r)$ is \emph{$(\alpha, c)$ \holder{}-thin}, or equivalently, $\bss{v, r}$ has \emph{$(\alpha, c)$ \holder{}-thin boundary}, if for every $\epsilon \leq 1$ we have
\begin{align*}
    \muss_v(\bss{\sss(v, r), \epsilon}) \leq c\epsilon^\alpha.
    % \,\,\,\,\, \text{ where }\,\,\, \bss{\sss(v, r), \epsilon} \coloneqq \{u \in \wss{v}: \dss_v(u, \sss(v, r)) \leq \epsilon \}.
\end{align*}
% where 
% $$ \bss{\sss(v, r), \epsilon} \coloneqq \{u \in \wss{v}: d_{\mtilde} (\pi(u), \pi(\sss(v, r))) \leq \epsilon \}$$

\begin{proposition} \label{rhc_cond}
    The following are equivalent.
    \begin{enumerate} [(i)]
        \item \label{rhc1} There exists a constant $\alpha > 0$ such that the following holds. For every $r_0 > 0$ there exists a constant $c = c(r_0)$ such that $\sss(v, r)$ is $(\alpha, c)$ \holder{}-thin for all $v \in T^1 \mtilde$ and $0 \leq r \leq r_0$.
        \item \label{rhc2} There exists a constant $\alpha' > 0$ such that the following holds. For every $r_0 > 0$ there exists a constant $c' = c'(r_0)$ such that for every $v \in T^1 \mtilde$, $\muss_v(\bss{v, r})$ is an $(\alpha', c')$ \holder{}-continuous function of $0 \leq r \leq r_0$.
        \item \label{rhc3} $(v, r) \mapsto \muss_v (\bss{v, r})$ is a \holder{}-continuous function on $T^1 \mtilde \times \RR^{\geq 0}$.
    \end{enumerate}
\end{proposition}

\begin{proof}
    $(\ref{rhc2}) \implies (\ref{rhc1})$ follows since
    $$\bss{\sss(v, r), \epsilon} \subset \bss{v, r + \epsilon} \setminus \mathring{\text{B}}^{\text{ss}} ({v, r - \epsilon}),$$
    for all $0 \leq r,\, v \in T^1 \mtilde$, and $\epsilon \leq 1$. $(\ref{rhc1}) \implies (\ref{rhc2})$ follows since there exists a constant $c_0 = c_0(r_0) > 0$ such that
    $$\bss{v, r + \epsilon} \setminus \bss{v, r - \epsilon} \subset \bss{\sss(v, r), c_0\epsilon},$$
    for all $0 \leq r \leq r_0,\, v \in T^1 \mtilde$, and $\epsilon \leq 1$. $(\ref{rhc3}) \implies (\ref{rhc2})$ is obvious, and $(\ref{rhc2}) \implies (\ref{rhc3})$ is proved in Lemma \ref{mub-holder}.
\end{proof}

\begin{definition} \label{rhc_def}
    If the items in Proposition \ref{rhc_cond} are satisfied, we say that the strongly stable Bowen-Margulis measure of balls is radius \holder{}-continuous, or in short, the RHC condition is satisfied.
\end{definition}

% \medskip
{\bf Comments on RC and RHC.} For this discussion, let $M$ be a manifold of variable negative curvature which is not necessarily compact. Fix $x, y \in \mtilde$, and let $\calB(T)$ denote the set of $\Gamma$--orbits of $y$ that lie in the ball of radius $T$ centered at $x$. If $M$ is $2$--dimensional and geometrically finite, the RC condition is satisfied since $\mu_o$ has no atoms, and RHC is satisfied by an argument similar to that of Theorem \ref{radius-hold-cont}. Hence, for the rest of this discussion, we may assume that $M$ is of dimension at least three. 

If $M$ is compact, asymptotics for $\#\calB(T)$ as $T$ goes to infinity was obtained by Margulis \cite{margulis-thesis} by considering a partition of $V \coloneqq S(x)$ into finitely many sets $V_i$, where each $V_i$ satisfies conditions (1)-(4) of Lemma 7.3 of that paper. (See Theorem 8 and the proof of Theorem 6 of that paper.) Condition (4) of that lemma is that the boundary of $V_i$ has measure zero (where the measure on $S(x)$ is the pullback of $\mu_x$ under the homeomorphism from $S(x)$ to $\bdry$ given by $u \mapsto u^+$). 
% This condition is similar the RC condition defined above. 
Choosing a partition that satisfies conditions (1)-(3) of that lemma and perturbing it slightly, we can always obtain a partition that also satisfies the fourth condition.

% Asymptotics for $\# \calB(T)$, as $T$ goes to infinity, was first obtained by Margulis when $M$ is compact, and later by Roblin, when $M$ is geometrically finite.
If $M$ is geometrically finite manifold of varibale negative curvature (and under the additional assumption that $\mu_o$ is finite), the asymptotics for $\#\calB(T)$ was obtained by Roblin \cite{roblin-thesis}. As mentioned earlier, our proof for Theorem \ref{control counting} resembles the proof of Theorem 4.1.1 of \cite{roblin-thesis}. In particular, our definition of a dynamical neighbourhood, given in Section \ref{test-intro}, is similar to \cite{roblin-thesis} (see the definition of $K(z, r)$ at page 58 of \cite{roblin-thesis}, also see Remark \ref{two-ham}), and in fact, Roblin faces a similar problem regarding the boundary of the strongly stable balls. Since \cite{roblin-thesis} is not concerned with error terms, it suffices that the RC condition is satisfied. If $M$ is geometrically finite and of constant negative curvature, the RC condition is proved in \cite[Lemma 1]{Rudolph} and \cite[Proposition 3.1]{roblin-rational}, but we could not understand either of these proofs. If $M$ is compact and of constant negative curvature, the RC condition is trivial, but we could not verify this condition for compact manifolds of variable negative curvature.
% We could only prove this under the additional assumptions of ??, and we are not sure if it holds in general.
In page 81 of \cite{roblin-thesis}, Roblin mentions that 
% la condition technique exigeant que le bord des boules instables soit n ́egligeable est acquise en courbure constante (voir [Ro2], proposition 3.1), tandis qu’autrement il est toujours possible de l’ ́eluder ` al’aidede fonctions plateaux.
``la condition technique exigeant que le bord des boules instables soit n {\'e}gligeable est acquise en courbure constante (voir \cite{roblin-rational}, proposition 3.1), tandis qu'autrement il est toujours possible de l' {\'e}luder {\`a}l'aidede fonctions plateaux.''
We do not know the exact method Roblin had in mind while writing those words.

% This idea carries over to the proof of Roblin as follows. 
Inspecting the proof of \cite[Theorem 4.1.1]{roblin-thesis}, we realize that it is not sensitive to the shape of the strongly stable sets $\bss{u, r}$, and it goes through as long as they have measure zero boundary. 
% For example, as mentioned in Remark \ref{two-ham}, it does not matter which distance we choose in the definition of $\bss{u, r}$.
This leads us to believe that (as in \cite{margulis-thesis}) fixing a partition $\bar{\calP} = \{\bar{V}_i\}$ of $\bdry$, where each $\bar{V}_i$ has measure zero boundary (with respect to the measure $\mu_o$, and hence $\mu_z$ for all $z \in \mtilde$), and then definining the sets $\bss{u, \bar{\calP}} \subset \wss{u}$ using this partition, one can circumvent the RC problem in Roblin. Assuming that the Bowen-Margulis measure is finite, such a partition can always be obtained by perturbing a given partition. To tackel the RHC problem using the same idea, we need for the boundary of $V_i$ to be \holder{}-thin. However, we do not know how to construct such partitions in general, even when $M$ is compact.  
% To circumvent this, we may change the partition $\calP_t$ as $t$ grows. Adopting this argument to the setting of Theorem ?? (or more generally, Theorem ?? of ??) has some other challenges, 

\medskip
{\bf Proof of RHC under additional assumptions.}
We use the following two theorems in the proof of Theorem \ref{radius-hold-cont}.

\begin{theorem} \label{jacobi horor} {\bf \cite[Theorem 2.4]{horosphere_geom}}
    Let $v \in T^1 \mtilde$ and $u \in S(\pi (v))$ be perpendicular to $v$. Let $J(t)$ denote the Jacobi field along $v_t$ such that $J(0) = u$ and the norm of $J(t)$ is bounded as $t$ goes to $-\infty$. Then for $t \geq 0$ we have
    $$\norm{J(t)} \leq e^{\clow t} \norm{u}.$$ 
\end{theorem}

 The following theorem is well-known. See \cite{sarnak_formula} for a discussion, and a stronger version.
\begin{theorem} \label{low_entrop}
    Let $M$ be an $n$--dimensional compact manifold with curvature bounded above by $-1$, and let $\delta$ denote the topological entropy of geodesic flow on the unit tangent bundle of $M$. Then we have
    $$\delta \geq n-1,$$
    and the equality holds if and only if the curvature is constantly equal to $-1$.
\end{theorem}

\nextalfa{rhc-hold}
For the following theorem, recall that $\clow$ is such that the curvature of $M$ is bounded below by $- \clow^2$.
\begin{theorem} \label{radius-hold-cont}
    Assume $M$ is two-dimensional, or $M$ is of dimension $n \geq 3$ and $\clow \leq \clowval$. Then there exists a constant $\alfa{\ref{rhc-hold}} > 0$ such that the following holds.
    For every $r_0 > 0$, there exists a constant $c = c(r_0)$ such that for every $v \in T^1 \mtilde$, $\mub{v, r}$ is an \hold{(\alfa{\ref{rhc-hold}},c)} function of $0 \leq r \leq r_0$.   
\end{theorem}

For the proof of this theorem, it is more natural to work with the strongly unstable foliation. For $v \in T^1 \mtilde$, define
$$\wuu(v) \coloneqq \{u \in T^1 \mtilde: u^- = v^-, \,\,\text{ and }\, \beta(v^-, \pi(v), \pi(u)) = 0\}.$$
For $r> 0$, $\duu_v, \suu(v, r), \buu(v, r), \muu_v$, and $\mubuu(v, r)$ can be defined as in the strongly stable case. The RHC (or RC) condition can also be written in terms of the strongly unstable foliation, and one can directly check that RHC (or RC) holds for the strongly stable foliation if and only if it holds for the strongly unstable foliation.

\begin{proof}
    % Since $\mub{v, r} = \mubuu(\iota(v), r)$, it is enough to prove the theorem for $\mubuu$ instead of $\mu_B^{\text{ss}}$. 
    By the above discussion, it is enough to verify item (\ref{rc1}) of Proposition \ref{rhc_cond} for the strongly unstable foliation. To this end, fix $r_0 > 0,\, 0 \leq r \leq r_0,\, v \in T^1 \mtilde$, and $\epsilon \ll 1$, and let $\varphi_i \from U_i \to V_i \subset \suu(v, r)$, for $i$ belonging to some finite set $I$, be charts on $\suu(v, r)$. Choose compact sets $K_i \subset U_i$ such that $\{\varphi_i(K_i)\}_{i \in I}$ covers $\suu(v, r)$, and for each $i \in I$, cover $K_i$ by $\bfO((\frac{1}{\epsilon})^{n -2})$ distinct cubes of the form 
    $$C \coloneqq [k_1 \epsilon, (k_1 + 1) \epsilon] \times \cdots \times [k_{n-2} \epsilon, (k_{n-2} + 1) \epsilon], \,\,\,\, \text{ where } k_j \in \ZZ \text{ for } 1 \leq j \leq n-2,$$ and denote the collection of these cubes by $\calK_i$.
    Let $\calA$ denote the collection of $\buu(\varphi_i(C), \epsilon)$ for $i \in I$ and $C \in \calK_i$.
    Then $\calA$ covers $\buu(\suu(v, r), \epsilon)$; $\#\calA = \bfO((\frac{1}{\epsilon})^{n -2})$, and $\diam^\text{uu}_v (A) = \bfO(\epsilon)$ for all $A \in \calA$, where $\diam^\text{uu}_v$ denotes the diameter with respect to the metric $\duu_v$.

    Let $\dhatuu_v$ denote the distance function on $\wuu(v)$ induced by restriction of the Riemannian structure of $\mtilde$ to its submanifold $\pi(\wuu(v))$. On $\buu(v, r_0 + 1)$, the metrics $\dhatuu_v$ and $\duu_v$ coincide upto a multiplicative constant that only depends on $r_0$. Hence, for every $A \in \calA$, we have $\diamhat^\text{uu}_v (A) = \bfO(\epsilon)$, where $\diamhat^\text{uu}_v $ denotes the diameter with respect to the metric $\dhatuu_v$.
    Let $T$ be such that $e^{-\clow T} = \epsilon$. Then by Theorem \ref{jacobi horor} we have $\diamhat^\text{uu}_v (\calG_T(A)) = \bfO(1)$, thus, the compactness of $M$ implies $\muu_{\calG_T(v)}(\calG_T(A)) = \bfO(1)$. Since the homeomorphism $u \mapsto u_T$ from $\wuu(v)$ to $\wuu(v_T)$ sends $\muu_v$ to $e^{-\delta T} \muu_{v_T}$, we have $\muu_v(A) = \bfO(e^{-\delta T})$. Since $A \in \calA$ was arbitrary, we obtain
    \begin{align} \label{muu bound}
        \muu_v(\buu(\suu (v, r), \epsilon)) \prec (\# \calA) e^{-\delta T} = \bfO(e^{(\clow(n -2) - \delta)T}).
    \end{align}
    If $n = 2$, then the above implies
    $$\muu_v(\buu(\suu (v, r), \epsilon)) = \bfO(\epsilon^\alpha)$$
    for $\alpha \coloneqq \frac{\delta}{\clow}$. This proves the theorem when $M$ is two-dimensional. If $n \geq 3$, by Theorem \ref{low_entrop}, either $M$ is of constant curvature $-1$, or otherwise $\delta > n -1$. In the former case the theorem follows directly, so we can assume we are in the latter case.
    If $K \leq \clowval$, then $\kappa \coloneqq \delta - \clow(n -2) > 0$, thus by (\ref{muu bound}) and the choice of $T$ we have 
    \begin{align*}
        \muu_v (\buu(\suu (v, r), \epsilon)) = \bfO(\epsilon^\alpha)
    \end{align*}
    for $\alpha \coloneqq \frac{\kappa}{\clow}$. Since the implicit constant in the above $\bfO$ can be chosen to only depend on $r_0$, the theorem is proved.
\end{proof}

\appendix

\section {A smoothing argument} \label{construct-chi}
%define \calB, the ball!

\nextalfa{app_F_holder}
In this section we prove the smoothing lemma (Lemma \ref{smoothing}) which is used in the proof of Lemma \ref{phi-mixing}. We adopt the notation introduced in Section \ref{test-intro}, and assume that $D$ is either a point $x \in \mtilde$, or equal to $L_\gamma$, the axis an element $\id \neq \gamma \in \Gamma$. The function $F \from \partial^1 D \to \RR$ is assumed to be \alfa{\ref{app_F_holder}}--\holder{}, and if $D = L$ we assume that $F$ is $\gamma$--invariant.
Given $u \in T^1 \mtilde$, recall that we defined $p_D^F(u)$ to be $\calG^F(p_D(u))$. Let $\pss_F(u)$ be the unique element of $\wss{p_D^F(u)}$ such that $\calG_t(\pss_F(u)) = u$ for some $t$, and set $\tauss(u) \coloneqq t$, that is, $\tauss(u)$ is such that
\begin{align} \label{def-of-tauss}
    \calG_{\tauss(u)} (\pss_F u) = u.
\end{align}
Finally, define $\rss (u) \coloneqq \dss_{\pdf u} (\pdf u, \pss_F u) = d_{\mtilde} \big(\pi(p^F_D u), \pi(\pss_F u)\big)$.

\nextalfa{tauss holder}
%prove \zeta, \eta (Hopf coordinates) are Holder!
\begin{lemma} \label{tauss is holder}
    There exists a constant $\alfa{\ref{tauss holder}}>0$ such that 
    $\tauss(u)$ and $\rss(u)$ are both $\alfa{\ref{tauss holder}}$--\holder{}.
\end{lemma}

\begin{proof}
    Given $u \in T^1 \mtilde$, by the definition of $\pss_F(u)$ we have
    \begin{align*}
        \beta_o (u^+, \pi(p_D^F u)) = \beta_o(u^+, \pi(\pss_F u)).
    \end{align*}
    By the defining equation for $\tauss$ (Equation (\ref{def-of-tauss})) we have
    \begin{align*}
        \beta_o(u^+, \pi(\pss_F u)) = \beta_o(u^+, \pi(u)) + \tauss(u).
    \end{align*}
    As a result,
    \begin{align*}
        \tauss(u) = \beta_o (u^+, \pi(p_D^F(u))) - \beta_o(u^+, \pi(u)).
    \end{align*}
    Since the right hand side of the above equation is a combination of \holder{}-continuous functions by the results in Section \ref{hold_section}, $\tauss$ is \holder{}--continuous as well.
    Plugging $\pss(u) = \calG_{-\tauss(u)} (u)$ into the defining equation for $\rss$, we have
    \begin{align*}
        \rss (u) = d_{\mtilde} \big(\pi(p^F_D (u)), \pi(\calG_{-\tauss(u)} (u))\big).
    \end{align*}
    As before, all the functions appearing on the right hand side are \holder{}--continuous, hence $\rss$ is \holder{}--continuous as well.
    % All the terms on the right hand side of the above are \holder{} functions of $u$, hence $\tauss$ is \holder{} as well.
\end{proof}

Let $(X, d)$ be a metric space. For $\epsilon > 0$ and a subset $A \subset X$, define the function $L_\epsilon (A) \from X \to \RR$ by
\begin{align*}
    L_\epsilon(A) (x) \coloneqq 1- \min \{ \frac{1}{\epsilon} d(x, A), 1 \}.
\end{align*}
Note that $L_\epsilon(A)$ is $\frac{1}{\epsilon}$--Lipschitz and it is constantly equal to $1$ on $A$. Denoting the $r$--neighbourhood of a set $S \subset X$ by $\calB_r(S)$, one can readily check that $\sprt(L_\epsilon(A)) = \overline{\calB}_\epsilon (A)$.

% Define
% \begin{align*}
%     \bfb_F (r, \tau) \coloneqq \thickbs{\tau}{ \calG^F(\partial^1 D), r}
% \end{align*}
For $r, \tau > 0$ set
\begin{align} \label{B-F in coordinates}
    \bfb_F (r, \tau) \coloneqq \{u \in T^1 \mtilde: \rss(u) \leq r , \text{ and }  |\tauss(u)| \leq \tau
      \},
\end{align}
and note that this set coincides with $\thickbs{\tau}{ \calG^F(\partial^1 D), r}$, defined in Section \ref{test-intro}.

\nextalfa{thick-alpha}
\nextconst{thick-const}
\begin{lemma} \label{neigh. of box}
    There are constants $\const{\ref{thick-const}}, \alfa{\ref{thick-alpha}} >0$ such that if $r$, $\tau$, and $\epsilon$ are all of order $1$, then we have
    \begin{align*}
        \calB_\epsilon \big(  \mathbf{B}_F (r, \tau) \big) \subset
        \mathbf{B}_F ( \rtaupm{+} ).
    \end{align*}
\end{lemma}

\begin{proof}
    Let $v \in \calB_\epsilon \big(  \mathbf{B}_F (r, \tau) \big)$. This means that there exists $u \in \mathbf{B}_F (r, \tau)$ with $d(u, v) < \epsilon$. 
    Since $u \in \mathbf{B}_F (r, \tau)$, by (\ref{B-F in coordinates}) we have $\rss(u) \leq r$ and $|\tauss(u)| \leq \tau$.
    By Lemma \ref{tauss is holder}, for $\alfa{\ref{thick-alpha}} \coloneqq \alfa{\ref{tauss holder}}$ and $\const{\ref{thick-const}} = \bfO(1)$ we have
    \begin{align*}
        |\rss(u) - \rss(v)| \leq \const{\ref{thick-const}} d(u, v)^\alfa{\ref{thick-alpha}} \leq \const{\ref{thick-const}} \epsilon^\alfa{\ref{thick-alpha}} 
        \implies 
        \rss(v) \leq r + \const{\ref{thick-const}}\epsilon^\alfa{\ref{thick-alpha}}.
    \end{align*}
    In a similar way, $|\tauss(v)| \leq \tau + \const{\ref{thick-const}}\epsilon^\alfa{\ref{thick-alpha}}$.
    The lemma now follows from (\ref{B-F in coordinates}).
\end{proof}

For $r, \tau > 0$, define the function $\varphi_F(r, \tau) \from T^1 \mtilde \to \RR$ by
\begin{align*}
    \varphi_F(r, \tau) (u) \coloneqq \frac{1}{
        2\tau \mub{\pdf(u), r}
    },
\end{align*}
and recall from Section \ref{test-intro} that $\test{D}{F}{r, \tau} (u) = \varphi_F(r, \tau) (u) \times \mathbbm{1} \big( \mathbf{B}_F(r, \tau) \big)$. Let $n_1 \coloneqq \frac{2}{\alfa{\ref{thick-alpha}}}$, and for $\tau \ll r <1$ and $\epsilon < \tau^{n_1}$ define 
\begin{align*}
    \chi_F{(r, \tau, \epsilon)} \coloneqq \test{D}{F}{r, \tau} \times L_\epsilon \big(
        \mathbf{B}_F(
            r - \const{\ref{thick-const}} \epsilon^{\alfa{\ref{thick-alpha}}}, 
            \tau - \const{\ref{thick-const}} \epsilon^{\alfa{\ref{thick-alpha}}}
        )
    \big).
\end{align*}
If $D = L$, then $\chi_F(r, \tau, \epsilon)$ is $\gamma$--invariant. Hence, using (\ref{pushdown-point}) or (\ref{pushdown-line}), depending on whether $D$ is a point or $D = L$, $\chi_F(r, \tau, \epsilon)$ descends to a function on $T^1 M$, which we denote by $\barchi_F(r, \tau, \epsilon)$.

\begin{lemma} {\bf (the smoothing lemma)} \label{smoothing}
    There are constants $\alfa{\ref{chi-holder}}$ and $\alfa{\ref{chiplus-alpha}}$, only depending on the \holder{} exponent of $F$, such that for all $\tau \ll r < 1$ and $\epsilon < \tau^{n_1}$, we have
    \begin{enumerate}
        \item $0 \leq \barchi_F(r, \tau, \epsilon) \leq \bartest{D}{F}{r, \tau}$.
        \item $\barchi_F(r, \tau, \epsilon)$ is \hold{\alfa{\ref{chi-holder}}} with $\alfa{\ref{chi-holder}}$--norm of order 
        $\frac{1}{\tau \epsilon}$.
        %  $\norm{\chi(r, \tau, \epsilon)}_{\alfa{\ref{chi-holder}}} = 
        % \bfO(\frac{1}{\tau \epsilon})$
        \item \label{lone_item} $\norm{\bartest{D}{F}{r, \tau} - \barchi_F(r, \tau, \epsilon)}_{L^1} =
        \bfO(\frac{
            \epsilon^{\alfa{\ref{chiplus-alpha}}}
        }{\tau} )$, where $\norm{\param}_{L^1}$ denotes the $L^1$--norm.
    \end{enumerate}
    % \begin{align*}
    %     \norm{\chi(r, \tau, \epsilon)}_{\alfa{\ref{chi-holder}}} = 
    %     \bfO(\frac{1}{\tau \epsilon}), \text{ and }\\
    %     \norm{\phi^i_F(r, \tau) - \chi_i(r, \tau, \epsilon )}_{L^1} \leq
    %     \bfO(\frac{
    %         \epsilon^{\alfa{\ref{chiplus-alpha}}}
    %     }{\tau} )
    % \end{align*}
\end{lemma}

The first item of the above lemma follows directly from the definition of $\barchi_F(r, \tau, \epsilon)$. The second item follows from Lemma \ref{chihodler-lem}, and the third item is proved at the end of this section.
% Note that by Lemma \ref{neigh. of box}, 
% \begin{align} \label{sprt of chiminus}
%     \sprt(\chi_F{(r, \tau, \epsilon)}) \subset \mathbf{B}_F(r, \tau)
% \end{align}

\nextalfa{chi-holder}

\begin{lemma} \label{chihodler-lem}
    There exists $\alfa{\ref{chi-holder}}$ such that for all $\tau \ll r <1$ and $\epsilon < \tau^{n_1}$, 
    $\chi(r, \tau, \epsilon)$ is 
    $(\alfa{\ref{chi-holder}}, \bfO(\frac{1}{\tau \epsilon}))$--\holder{}.
    % \begin{align*}
    %     \norm{\chi(r, \tau, \epsilon)}_\alfa{\ref{chi-holder}}} 
    %     = \bfO(\frac{1}{\tau \epsilon})
    % \end{align*} 
\end{lemma}

\begin{proof}
    By Lemma \ref{neigh. of box}, the support of $L_\epsilon\big(\mathbf{B}_F (\rtaupm{-})\big)$ is a subset of $\bfb_F (r, \tau)$, hence we have
    \begin{align} \label{chi in varphi}
        \chi_F{(r, \tau, \epsilon)} (u) = \varphi_F(r, \tau)(u) \times  L_\epsilon\big(\mathbf{B}_F (\rtaupm{-})\big).
    \end{align}
    The function $u \mapsto \mub{\pdf(u), r}$ is a combination of \holder{}-continuous functions (the fact that $\mub{\param, \param}$ is \holder{}-continuous is proved in Lemma \ref{mub-holder}), hence it is \hold{\alpha} for some constant $\alpha > 0$. This makes $\varphi_F(r, \tau)$ a \hold{(\alpha, \bfO(\frac{1}{\tau}))} function.
    As mentioned at the beginning of the section, $L_\epsilon(A)$ is \lip{\frac{1}{\epsilon}} for every subset $A$ of a general topological space $X$, hence $ L_\epsilon\big(\mathbf{B}_F (\rtaupm{-})\big)$ is \lip{\frac{1}{\epsilon}}. 
    It follows that $\chi_F{(r, \tau, \epsilon)}$ is $(\alfa{\ref{chi-holder}}, \bfO(\frac{1}{\tau \epsilon}))$--\holder{} for $\alfa{\ref{chi-holder}} \coloneqq \alpha$ .
\end{proof}

\begin{lemma} \label{mub-holder}
    The function 
    $\mub{u, r} \from T^1 \mtilde \times [0, \infty) \to \RR$ is \holder{}--continuous. 
    % \hold{\alfa{\ref{mub-holder}}}
\end{lemma}

\begin{proof}
    Fixing $u\in T^1 \mtilde$, $\mub{u, \param}$ is \holder{}--continuous by Theorem \ref{radius-hold-cont}, hence we only need to show that $\mub{\param, r}$ is \holder{}--continuous for a fixed $r$. This is essentialy proved in Lemma 11 of \cite{PPCommonPerp}, thus we only give an outline of the proof.
    
    Fix $r$, and let $u,v \in T^1 \mtilde$ be close by. Define the function $H^{uv} \from \bss{u, r} \to \wss{v},\, u' \mapsto v'$ in a way that $(u')^{-} = (v')^{-}$. This function satisfies the following.
    \begin{enumerate}
        \item There exists $\alpha$ such that $d\big(\pi(u'), \pi(H^{uv} u')\big) = \bfO (d(u, v)^\alpha)$ for $u' \in \bss{u, r}$.
        \item  For $u' \in \bss{u, r}$,
        \begin{align*}
            d \muss_v(v') = \big(1 + \bfO(d(u, v)^\alpha)\big)\, d\muss_u(u'). 
        \end{align*}
    \end{enumerate}
    % We can show that  Using this, and 
    By the first item and triangle inequality, if we choose $\epsilon$ to be a large multiple of $d(u, v)^\alpha$, then $H^{uv}(\bss{u, r - \epsilon}) \subset \bss{v, r}$, thus
    \begin{align*}
        \mub{v, r} \geq
        \muss_v \big(H^{uv}(\bss{u, r - \epsilon})\big) = \big(1 + \bfO(d(u, v)^\alpha)\big)\, \mub{u, r - \epsilon} \tag{by item (2)}\\ 
        = \big(1 + \bfO(d(u, v)^{\alpha \alfa{\ref{rhc-hold}}})\big)\, \mub{u, r} \tag{by Theorem \ref{radius-hold-cont}}.
    \end{align*}
    Therefore, letting $\alpha' \coloneqq \alpha \alfa{\ref{rhc-hold}}$, there exists a constant $c$ such that
    \begin{align*}
        \mub{v, r} \geq \mub{u, r} - c d(u, v)^{\alpha'}.
    \end{align*}
    Changing the role of $u$ and $v$ in the above argument, we deduce that $\mub{\param, r}$ is $\alpha'$--\holder{}.
\end{proof}

\nextalfa{chiplus-alpha}
\nextconst{chiplus const}
\begin{lemma} \label{chiminus lowerbound}
    There are constants $\const{\ref{chiplus const}}, \alfa{\ref{chiplus-alpha}}$ such that for $\tau \ll r < 1$ and $\epsilon < \tau^{n_1}$ we have
    \begin{align} \label{chi- lowerbound eq.}
        (1 - \const{\ref{chiplus const}} \frac{\epsilon^\alfa{\ref{chiplus-alpha}}}{\tau})\,
        % \big( 1 - \chilowerdelt
        % \big) \,
        \test{D}{F}{r, \tau}
        \leq \chi_F ( \rtaupm{+}, \epsilon ).
    \end{align} 
\end{lemma}

\begin{proof}
    Let $r' \coloneqq r + \const{\ref{thick-const}} \epsilon^\alfa{\ref{thick-alpha}}$ and $\tau' \coloneqq \tau + \const{\ref{thick-const}} \epsilon^\alfa{\ref{thick-alpha}}$.
    % Writing (\ref{chi in varphi}) with $r, \tau$ replaced by $\rtaupm{+2}$ we have
    % \begin{align*}
    %     \chi(\rtaupm{+2}, \epsilon) = \varphi_F(\rtaupm{+2}) \times L_\epsilon(\mathbf{B}_F(r, \tau))
    % \end{align*}
    Since $\sprt(\test{D}{F}{r, \tau}) = \mathbf{B}_F(r, \tau)$ and $\chi(r', \tau', \epsilon)$ is equal to $\varphi_F(r', \tau')$ on $\bfb_F(r, \tau)$, 
    % the above equation implies that for every $u \in \sprt(\test{D}{F}{r, \tau})$ we have
    % \begin{align*}
    %     \chi(\rtaupm{+2}, \epsilon) = \varphi_F(\rtaupm{+2}), 
    %     % \\ \text{ for every } u \in \sprt(\test{D}{F}{r, \tau})
    % \end{align*}
    % for every $u \in \sprt(\test{D}{F}{r, \tau})$ we have
    % $\chi(\rtaupm{+2}, \epsilon) = \varphi_F(\rtaupm{+2})$,
    % Since $\test{D}{F}{r, \tau}$ is equal to $\varphi_F(r, \tau)$ on its support, to prove the lemma 
    it is enough to show that for $u \in \bfb_F(r, \tau)$ we have
    \begin{align*}
        \big(1 - \bfO \big(
            \const{\ref{thick-const}}
            (\frac{\epsilon^\alfa{\ref{thick-alpha}}}{\tau} +
            \epsilon^{
                \alfa{\ref{rhc-hold}} \alfa{\ref{thick-alpha}}
            })
        \big) \big) \,
        \varphi_F(r, \tau) \leq \varphi_F(r', \tau'),
    \end{align*}
    where $\alfa{\ref{rhc-hold}}$ is the constant given by Theorem \ref{radius-hold-cont}. (then we take $\alfa{\ref{chiplus-alpha}} \coloneqq \alfa{\ref{rhc-hold}} \alfa{\ref{thick-alpha}}$ and $\const{\ref{chiplus const}} \coloneqq \bfO(\const{\ref{thick-const}})$.)
    This follows by writing the defining equation for $\varphi_F(\param, \param)$ on both sides of the above and using the  equalities
    \begin{align*}
        \frac{\tau'}{\tau} = 1 + \bfO(
            \const{\ref{thick-const}}  
            \frac{\epsilon^\alfa{\ref{thick-alpha}}}{\tau}   
        ),\,\,\,\,\, \text{ and } \,\,\,\,\,
        \frac{\mub{\pdf u, r'}}{\mub{\pdf, r}} = 1 + \bfO(
            \const{\ref{thick-const}} \epsilon^{
                \alfa{\ref{rhc-hold}} \alfa{\ref{thick-alpha}}
            }
        ).
    \end{align*}
\end{proof}
%everything invariant and descend downstairs!

\begin{proof} [Proof of item (\ref{lone_item}) of Lemma \ref{smoothing} ]
    Let $r' \coloneqq r - \const{\ref{thick-const}} \epsilon^\alfa{\ref{thick-alpha}}$ and $\tau' \coloneqq \tau - \const{\ref{thick-const}} \epsilon^\alfa{\ref{thick-alpha}}$, and
    recall that by Proposition \ref{integral-of-test},
    \begin{align} \label{sigma-equals}
        \sigma(F) = \int_{T^1 M} \bartest{D}{F}{r, \tau} \,d\mbarBM = \int_{T^1 M} \bartest{D}{F}{r', \tau'} \,d\mbarBM,
    \end{align}
    where $\sigma(F) \coloneqq \sigma_\gamma(F)$ (resp. $\sigma(F)\coloneqq \sigma_x(F)$) when $D = L_\gamma$ (resp. $D = x$). By Lemma \ref{chiminus lowerbound}, 
    \begin{align*}
        ( 1 - \chilowerdelt )\, \bartest{D}{F}{r', \tau'} \leq \barchi_F(r, \tau, \epsilon).
    \end{align*}
    By integrating and using (\ref{sigma-equals}) we obtain
    % Taking integral from both sides of (\ref{chi- lowerbound eq.}) and using the above, we obtain
    \begin{align*}
        \sigma(F) - \int_{T^1 M} \barchi_F(r, \tau, \epsilon) = \bfO(\chilowerdelt).
    \end{align*}
    Since $ 0 \leq \barchi_F(r, \tau, \epsilon) \leq \bartest{D}{F}{r, \tau}$, we have
    \begin{align*}
        \norm{\bartest{D}{F}{r, \tau} - \barchi_F (r, \tau, \epsilon)}_{L^1} =
        % \int_{T^1 M} |\test{D}{F}{r, \tau} - \barchi_F (r, \tau, \epsilon)| = 
        \int_{T^1 M} \bartest{D}{F}{r, \tau} - \int_{T^1 M} \barchi_F (r, \tau, \epsilon) = 
        \sigma(F) - \int_{T^1 M} \barchi_F (r, \tau, \epsilon) =
        \bfO(\frac{
            \epsilon^{\alfa{\ref{chiplus-alpha}}}
        }{\tau} ).
    \end{align*}
\end{proof}

%%%%%%%%%%%%%%%%%%%%%%%%%%%%%%

\bibliographystyle{alpha}
\bibliography{references}

%%%%%%%%%%%%%%%%%%%%%%%%%%%%%%%%%

\end{document}